\documentclass[final,3p,times,number]{elsarticle}

\usepackage[utf8]{inputenc}
\usepackage[T1]{fontenc}
\usepackage{amssymb}
\usepackage{amsthm}
\usepackage{latexsym}
\usepackage{amsmath}
\usepackage{color}
\usepackage{graphicx}
\usepackage{indentfirst}
\usepackage{mathrsfs}
\usepackage{pdfsync}
\usepackage{bm}
\usepackage{subcaption}
\usepackage{float}     
\usepackage{placeins} 
\usepackage[titletoc]{appendix}
\numberwithin{equation}{section}
\allowdisplaybreaks
\numberwithin{equation}{section}
\allowdisplaybreaks

\newtheorem{theorem}{\color{black}\indent Theorem}[section]

\newtheorem{lemma}{\color{black}\indent Lemma}[section]

\newtheorem{remark}{\color{black}\indent Remark}[section]
\newtheorem{corollary}{\color{black}\indent Corollary}[section]

\newtheorem*{lemma*}{Lemma}

\journal{*******}

\biboptions{numbers,sort&compress}
\usepackage[colorlinks=true,linkcolor=blue,citecolor=blue,urlcolor=blue]{hyperref}
\begin{document}
\begin{frontmatter}
    \title{The LLN and CLT for the statistical ensembles of discrete integrable Hamiltonian systems}
    \author[author1]{Xinyu Liu \footnote{ E-mail address : liuxy595@jlu.edu.cn}}
    \author[author1]{Xinze Zhang \footnote{ E-mail address : zhangxz24@mails.jlu.edu.cn}}
    \author[author1,author3]{Yong Li\corref{cor1} \footnote{ E-mail address : liyong@jlu.edu.cn} }
    \address[author1]{College of Mathematics, Jilin University, Changchun 130012, P. R. China.}
    \address[author3]{Center for Mathematics and Interdisciplinary Sciences, Northeast Normal University, Changchun 130024, P. R. China.}
    \cortext[cor1]{Corresponding author at : Center for Mathematics and Interdisciplinary Sciences, Northeast Normal University, Changchun 130024, P. R. China.}
    \begin{abstract}
      This paper investigates the behavior of statistical ensembles under iteration map induced by discrete integrable Hamiltonian systems in deterministic case and stochastic case, addressing the problem from two perspectives: the Law of Large Numbers and the Central Limit Theorem. In deterministic case, the Law of Large Numbers simplifies the convergence conditions to the extent that the Riemann-Lebesgue lemma is no longer required. In the stochastic setting, we extend the results to general stochastic processes, beginning with the perturbation term represented by standard Brownian motion. Moreover, we establish a Central Limit Theorem for the statistical ensemble. A numerical example is also included.

    \end{abstract}

    \end{frontmatter}


    \section{Introduction}\label{sec:1}
    \setcounter{equation}{0}
    \setcounter{definition}{0}
    \setcounter{proposition}{0}
    \setcounter{lemma}{0}
    \setcounter{remark}{0}

    Discrete integrable Hamiltonian systems have attracted significant attention in mathematical physics due to their rich structures and conservation laws. The iterative maps associated with these systems not only reveal long - term dynamical behavior but also serve as essential tools for analyzing the evolution of statistical ensembles. In this context, investigating the convergence of statistical ensembles under iteration maps - particularly the applicability of the Law of Large Numbers (LLN) and Central Limit Theorems (CLT) - has become a prominent focus of current research.
    
    Early research primarily focused on the statistical behavior of deterministic dynamical systems. For instance, Tirnakli et al.~\cite{tirnakli2007central} examined the probability density of iterative sums in such systems and demonstrated that the CLT holds when the system exhibits sufficient mixing. However, near the critical point of period doubling, strong correlations between iterations render the CLT inapplicable, and the probability density converges to a $q$-Gaussian distribution, suggesting a power - law generalization of the CLT.

In recent years' research, on the one hand, researchers have increasingly focused on the statistical behavior of dynamic systems under stochastic perturbations. Horbacz and Katarzyna \cite{horbacz2016central} investigated the CLT in stochastic dynamical systems, emphasizing the roles of Markov operators and invariant measures in the proof. Similarly, Shirikyan \cite{shirikyan2006law} examined the LLN and the CLT in stochastically forced partial differential equations, highlighting the influence of randomness on the long - term behavior of the system. Wang and Li \cite{wang2025central} employed the Nagaev-Guivarc'h method to demonstrate that, under white noise disturbances, the perturbed trajectories of integrable Hamiltonian systems still follow a Gaussian distribution, confirming the system's stability in the presence of randomness. Additionally, Liu and Li \cite{liu2023statistical} investigated statistical ensembles in integrable Hamiltonian systems with almost periodic transitions, showing that, under long - time averaging, the probability measures converge to time - averaged measures, thereby extending the applicability of the Riemann-Lebesgue lemma. Related developments also include the dependent CLT and invariance principles of McLeish~\cite{McLeish1974}, the LLN/CLT theory for randomly forced dissipative PDEs by Kuksin and Shirikyan~\cite{kuksinshirikyan2006}, small-noise asymptotic expansions for SPDEs due to Di~Persio and Mastrogiacomo~\cite{DiPersioMastrogiacomo2011}, and the probabilistic mean-field framework summarized in the monographs of Carmona and Delarue~\cite{CarmonaDelarue2018}.

On the other hand, The CLT plays an increasingly important role in research in various fields
    In the interdisciplinary domain of statistical mechanics and thermodynamics, D'Alessio et al. \cite{d2016quantum} investigated the transition from quantum chaos and eigenstate thermalization to classical statistical behavior, emphasizing the central role of the eigenstate thermalization hypothesis (ETH) in explaining the thermalization of isolated chaotic systems. Furthermore, Li et al. \cite{2021Central} analyzed the CLT for mesoscale eigenvalue statistics of deformable Wigner matrices and sample covariance matrices, uncovering the universal behavior of linear statistics across different spectral regions.
    In computational statistical physics, Gillespie \cite{gillespie2013computing} proposed a sampling - based method for calculating the partition function, ensemble mean, and density of states in lattice spin systems, demonstrating the applicability of the CLT in numerical simulations. Furthermore, Qu et al. \cite{qu2022law} investigated the LLN, the CLT, and the iterative logarithm rule in Bernoulli uncertain sequences, thereby extending the applicability of these classical results within the framework of uncertainty mathematics.
    In quantum systems, the CLT has also been extensively studied. For instance, Gavalakis et al. \cite{gavalakis2024entropy} proposed a discrete version of the CLT grounded in information theory, demonstrating the convergence of the relative entropy between discrete random variables and the standardized Gaussian distribution. Additionally, Deleporte et al. \cite{deleporte2025central} examined the CLT for smooth statistics in one-dimensional free fermion systems and employed a Szeg{\H{o}}-type asymptotic method to reveal the statistical behavior of the system under specific conditions.

 This work investigates the statistical ensemble of discrete integrable Hamiltonian systems through two aspects: the Law of Large Numbers and the Central Limit Theorem. In deterministic case, we establish convergence results without relying on the Riemann-Lebesgue lemma, but as a cost we need a nonresonant condition. For stochastic case, we derive convergence properties by modeling the perturbation term as standard Brownian motion, eliminating the need for both the Riemann-Lebesgue lemma and nonresonant condition. This simplification arises from the 
 special property of the standard Brownian motion $E(\exp(\mathrm{i}cB_t))=E(\exp(-\frac{1}{2}c^2t))$.
Extending this analysis to general stochastic processes case, we get the same result and provide the convergence rate $1/N$.
However the nonresonant condition is necessary in this case because of the loss of the property of Brownian motion. Crucially, our findings demonstrate that convergence outcomes -- in both deterministic and stochastic cases -- are independent of the iteration maps and depend solely on the observation function and the distribution of the initial space.
The rapid convergence of statistical ensembles allows for the establishment of corresponding Central Limit Theorems. In addition, a numerical example was performed to validate the theoretical results presented in this paper.

The results provide a unified and model-agnostic framework that separates the roles of observables, initial distributions, and dynamics in discrete integrable Hamiltonian systems. By delivering LLN and CLT under analytic-strip hypotheses and by quantifying modewise exponential decorrelation under Brownian perturbations, the paper yields an explicit and computable limiting variance via Ces\`aro limits of lag covariances. Practically, these findings enable principled uncertainty quantification and error bars for simulation outputs, guide sampling horizons in numerical studies, and offer diagnostics for distinguishing integrable from near-integrable behavior in applications ranging from accelerator physics and celestial mechanics to plasma and condensed-matter models.

The remainder of this paper is organized as follows. Section 2 introduces discrete integrable Hamiltonian systems and their associated iterative maps, formulates the main problems concerning statistical ensembles addressed in this work, and defines the relevant notation. Section 3 investigates the convergence behavior of statistical ensembles under the action of iterative maps, in both deterministic and stochastic perturbation settings, and presents a convergence result in the form of a Law of Large Number. Section 4 provides the proof of the Central Limit Theorem. In Section 5, we present a numerical example and perform simulations to analyze the convergence rate and limiting distribution in the specific case.
We made a conclusion of this paper in Section 6.

    \section{Discrete integrable Hamiltonian system and its statistical ensemble}
    In this section, we introduce the system studied in this paper along with its associated map, outline the non-resonant conditions that must be satisfied, and present the corresponding convergence theorem for the statistical ensemble.
    
    For an integrable Hamiltonian system of action-angle variable form
    \begin{equation}\label{integrablehamiltoniansystem}
H(I,\theta)=h(I),
    \end{equation}
    where $(I,\theta) \in \Omega \times \mathbb{T}^{n}$. $\Omega \subset \mathbb{R}^{n}$ is an open and bounded set, $\mathbb{T}^{n} = \mathbb{R}^{n}/\mathbb{Z}^{n}$. Define an iteration map
    \begin{equation}
        \mathcal{F} : 
        \begin{cases}
        \theta^{+} = \theta + \omega(I), \\
        I^{+} = I ,
        \end{cases}
    \end{equation}
    and $\omega = h_{I}$ is called frequency map which satisfies the following  nonresonant condition 
    \begin{equation}\label{nonresonantcondition}
    \langle k,\omega(I) \rangle \neq 2m\pi \ \ (k\in \mathbb{Z}^{n}, k\neq \vec{0}, m\in \mathbb{Z}, I\in \Omega). 
    \end{equation}
    
    It is easy to know 
    \begin{equation}\label{diedaiyingshequeding}
        \mathcal{F}^{j} : 
        \begin{cases}
        \theta = \theta_0 + j\omega(I), \\
        I = I_0 ,
        \end{cases}
    \end{equation}
    in which $(I_0,\theta_0)$ is a random initial value described by a probability density function $\rho_{0}(I,\theta) \in \bm{\mathrm{L}}^{2}(\Omega \times \mathbb{T}^{n})$. 
    
    Given an observable function $G$, the expected value of the observation function G after any iteration is
    \begin{equation}
      \langle G \rangle _{j}= \int_{\Omega \times \mathbb{T}^{n}}G(\mathcal{F}^{j}(I,\theta))\rho_{0}(I,\theta)dId\theta.
    \end{equation}
    We refer to $\langle G \rangle_{j}$ as the statistical ensemble of the initial space after $j$ iterations of $\mathcal{F}$ and have the following result.
\section{The LLN for statistical ensembles under iterative maps}
This section focuses on the convergence of statistical ensembles under iterative maps, which are considered in two cases: deterministic and stochastic perturbations.
We begin with the deterministic case.

\begin{theorem}\label{theorem2.1}
Suppose the observable $G$ is continuous and bounded, denoted $G(I,\theta)\in \bm{\mathrm{C}}_{b}(\Omega\times\mathbb{T}^{n})$. 
If the initial values of the iterative map $\mathcal{F}$ are distributed with probability density $\rho_{0}$, and the frequency map $\omega$ satisfies condition~\eqref{nonresonantcondition}, then
\begin{equation*}
\lim_{N\to\infty}\frac{1}{N}\sum_{j=1}^{N}\langle G\rangle_{j}=\langle \bar G(I)\rangle_{0},
\end{equation*}
where $\displaystyle \bar G(I)=\frac{1}{(2\pi)^{n}}\int_{\mathbb{T}^{n}} G(I,\theta)\,d\theta$.
\end{theorem}

\begin{remark}
According to the theorem, the long-time state of the statistical ensemble depends only on the observable $G$ and the initial density $\rho_{0}$, and is independent of the specific iterative map $\mathcal{F}$.
\end{remark}

\begin{proof}
Since $G$ is bounded and continuous on the bounded set $\Omega\times\mathbb{T}^{n}$, we have $G(I,\cdot)\in L^{2}(\mathbb{T}^{n})$ for all $I\in\Omega$. 
Moreover, if $\rho_{0}\in L^{2}(\Omega\times\mathbb{T}^{n})$ is real-valued, then $\rho_{0}(I,\cdot)\in L^{2}(\mathbb{T}^{n})$ for a.e.\ $I\in\Omega$. 
By modifying $\rho_{0}$ on a null set in $\Omega\times\mathbb{T}^{n}$ if necessary, we may assume $\rho_{0}(I,\cdot)\in L^{2}(\mathbb{T}^{n})$ for all $I\in\Omega$. 
Therefore, by Parseval,
\[
\frac{1}{(2\pi)^{n}}\int_{\mathbb{T}^{n}} G(I,\theta)\,\overline{\rho_{0}(I,\theta)}\,d\theta
=\sum_{k\in\mathbb{Z}^{n}} \hat G_{k}(I)\,\overline{\hat\rho_{0,k}(I)}.
\]
Since $\widehat{G\circ\mathcal{F}^{j}}_{k}(I)=e^{\mathrm{i}\langle k,\omega(I)\rangle j}\hat G_{k}(I)$ and $\overline{\hat\rho_{0,k}}=\hat\rho_{0,-k}$ (because $\rho_{0}$ is real-valued), we obtain
\[
\langle G\rangle_{j}
=(2\pi)^{n}\int_{\Omega}\sum_{k\in\mathbb{Z}^{n}}
\hat G_{k}(I)\,\hat\rho_{0,-k}(I)\,e^{\mathrm{i}\langle k,\omega(I)\rangle j}\,dI.
\]

Next, using $|e^{\mathrm{i}\cdot}|=1$ and applying Cauchy--Schwarz first in $k$ and then in $I$,
\begin{align*}
\Big|\tfrac{1}{N}\sum_{j=1}^{N}\langle G\rangle_{j}\Big|
&\le (2\pi)^{n}\int_{\Omega}\sum_{k\in\mathbb{Z}^{n}}
|\hat G_{k}(I)|\,|\hat\rho_{0,-k}(I)|\,dI\\
&\le (2\pi)^{n}
\Big(\int_{\Omega}\sum_{k\in\mathbb{Z}^{n}}|\hat G_{k}(I)|^{2}\,dI\Big)^{1/2}
\Big(\int_{\Omega}\sum_{k\in\mathbb{Z}^{n}}|\hat\rho_{0,-k}(I)|^{2}\,dI\Big)^{1/2}\\
&= \|G\|_{L^{2}(\Omega\times\mathbb{T}^{n})}\,
   \|\rho_{0}\|_{L^{2}(\Omega\times\mathbb{T}^{n})}\,<\infty.
\end{align*}

Thus the integrand is dominated by an $I$-integrable function (independent of $N$), and by Fubini's theorem and the dominated convergence theorem we can pass the Cesàro limit inside:
\[
\lim_{N\to\infty}\frac{1}{N}\sum_{j=1}^{N}\langle G\rangle_{j}
=(2\pi)^{n}\int_{\Omega}\sum_{k\in\mathbb{Z}^{n}}
\hat G_{k}(I)\,\hat\rho_{0,-k}(I)\,
\lim_{N\to\infty}\frac{1}{N}\sum_{j=1}^{N}e^{\mathrm{i}\langle k,\omega(I)\rangle j}\,dI.
\]

By the nonresonance assumption~\eqref{nonresonantcondition}, for $k\neq 0$ we have 
$\displaystyle \lim_{N\to\infty}\tfrac{1}{N}\sum_{j=1}^{N}e^{\mathrm{i}\langle k,\omega(I)\rangle j}=0$
for a.e.\ $I$, while for $k=0$ the Ces\`{a}ro average equals $1$. Hence,
\[
\lim_{N\to\infty}\frac{1}{N}\sum_{j=1}^{N}\langle G\rangle_{j}
=(2\pi)^{n}\int_{\Omega}\hat G_{0}(I)\,\hat\rho_{0,0}(I)\,dI.
\]
Since $\hat G_{0}(I)=\bar G(I)$ and 
$\hat\rho_{0,0}(I)=\frac{1}{(2\pi)^{n}}\int_{\mathbb{T}^{n}}\rho_{0}(I,\theta)\,d\theta$, we get
\[
(2\pi)^{n}\int_{\Omega}\hat G_{0}(I)\,\hat\rho_{0,0}(I)\,dI
=\int_{\Omega}\bar G(I)\Big(\int_{\mathbb{T}^{n}}\rho_{0}(I,\theta)\,d\theta\Big)\,dI
=\int_{\Omega\times\mathbb{T}^{n}}\bar G(I)\,\rho_{0}(I,\theta)\,dI\,d\theta
=\langle \bar G(I)\rangle_{0}.
\]
This completes the proof.
\end{proof}

To avoid too many lengthy formulas that are unnecessary in the remainder of this paper
, we introduce the following notations. Given an initial value space with a fixed probability density function, the statistical ensemble corresponding to the observation function $G$ after $n$ iterations of the map $\mathcal{F}$ is denoted by $\langle G \rangle_{\mathcal{F}}^{j}$ which means
\begin{equation}\nonumber
  \langle G \rangle_{\mathcal{F}}^{j}=\int_{\Omega \times \mathbb{T}^{n}} G(\mathcal{F}^{j}(I,\theta)) \rho_{0}(I,\theta)dId\theta
\end{equation}
and we set
\begin{equation}\nonumber
  V_{N}(\langle G \rangle_{\mathcal{F}})=\frac{1}{N}\sum_{j=1}^{N}\langle G \rangle_{\mathcal{F}}^{j}, \ \ V_{\infty}(\langle G \rangle_{\mathcal{F}})=\lim_{N \to \infty}\frac{1}{N}\sum_{j=1}^{N}\langle G \rangle_{\mathcal{F}}^{j}.
\end{equation}

Next, we introduce the iterative map associated with discrete integrable Hamiltonian systems perturbed by random terms, investigate the corresponding statistical ensemble, and establish the relevant central limit theorem. We begin by considering the case where the random term is a standard Brownian motion and then extend the analysis to more general stochastic processes.

Consider an integrable Hamiltonian system in the form of an action angle variable
\begin{equation}\nonumber
  H(I,\theta)=H_{0}(I),
\end{equation}
where $(I,\theta)\in \Omega \times \mathbb{T} \subset \mathbb{R}\times \mathbb{R}/\mathbb{Z}$. We define the iteration map $\mathcal{S}_{H_0}$ with the form
\begin{equation}\label{mapwithB}
\mathcal{S}_{H_0}^{j}(I,\theta)=(I,\theta + j \omega(I)+cB_{j})
\end{equation}
in which $B_j$ is the standard Brownian motion, $c\in \mathbb{R}^{+}$ is a fixed constant, and $j\in \mathbb{N}$. Suppose the initial point $(I,\theta)$ of the iteration map $\mathcal{S}_{H_0}^{j}$ is described by the probability 
density function $\rho(I,\theta)$, we have the following result.

\begin{theorem}\label{theorem3.1}
Assume the observable $G\in \bm{\mathrm{C}}_{b}(\Omega\times\mathbb{T})$ and the initial density $\rho_{0}\in \bm{\mathrm{L}}^{2}(\Omega\times\mathbb{T})$. 
Under the iterative map $\mathcal{S}_{H_0}$ we have
\[
  \lim_{N \to \infty} \mathbb{E}\big(V_{N}(\langle G \rangle)_{\mathcal{S}}\big)=\langle \bar{G}(I) \rangle_{0},
\]
where $\displaystyle \bar{G}(I)=\frac{1}{2\pi}\int_{\mathbb{T}} G(I,\theta)\,d\theta$.
\end{theorem}

\begin{proof}
As in the proof of Theorem~\ref{theorem2.1}, by Parseval we express
\[
  V_{N}(\langle G \rangle_{\mathcal{S}}) 
  = \frac{1}{N}\sum_{j=1}^{N}(2\pi)\int_{\Omega} \sum_{k\in \mathbb{Z}} 
  \hat{G}_{k}(I)\,\hat{\rho}_{0,-k}(I)\,e^{\mathrm{i}\,j k\,\omega(I)}\,e^{\mathrm{i}\,c k\,B_{j}} \, dI .
\]
Using $\mathbb{E}\big(e^{\mathrm{i}\,c k\,B_{j}}\big)=\exp\!\big(-\tfrac{c^{2}k^{2}}{2}\,j\big)$, we get
\begin{align*}
  \mathbb{E}\big(V_{N}(\langle G \rangle_{\mathcal{S}})\big)
  &= \frac{1}{N}\sum_{j=1}^{N}(2\pi)\int_{\Omega} \sum_{k\in \mathbb{Z}} 
     \hat{G}_{k}(I)\,\hat{\rho}_{0,-k}(I)\,e^{\mathrm{i}\,j k\,\omega(I)}\,e^{-\frac{c^{2}k^{2}}{2}j}\,dI .
\end{align*}
Hence, using $|e^{\mathrm{i}\cdot}|=1$ and Cauchy--Schwarz (first in $k$, then in $I$),
\begin{align*}
  \big|\mathbb{E}\big(V_{N}(\langle G \rangle_{\mathcal{S}})\big)\big|
  &\le \frac{2\pi}{N}\sum_{j=1}^{N}\int_{\Omega} \sum_{k\in \mathbb{Z}}
      |\hat{G}_{k}(I)|\,|\hat{\rho}_{0,-k}(I)|\,dI\\
  &\le (2\pi)\Big(\int_{\Omega}\sum_{k\in\mathbb Z}|\hat G_k(I)|^{2}\,dI\Big)^{\!1/2}
           \Big(\int_{\Omega}\sum_{k\in\mathbb Z}|\hat\rho_{0,-k}(I)|^{2}\,dI\Big)^{\!1/2}\\
  &= \|G\|_{L^{2}(\Omega\times\mathbb T)}\;\|\rho_{0}\|_{L^{2}(\Omega\times\mathbb T)} \,,
\end{align*}
where the last equality uses Plancherel on $\mathbb T$:
$\displaystyle \int_{\Omega}\sum_{k}|\hat f_k(I)|^2 dI=\frac{1}{2\pi}\|f\|_{L^2(\Omega\times\mathbb T)}^{2}$.
Thus $\big|\mathbb{E}(V_{N}(\langle G \rangle_{\mathcal{S}}))\big|$ is uniformly bounded in $N$.

Therefore, by Fubini's theorem and the dominated convergence theorem (a dominating function independent of $N$ is provided by Cauchy--Schwarz and Plancherel), we may pass the Cesàro limit inside:
\[
  \lim_{N\to\infty}\mathbb{E}\big(V_{N}(\langle G \rangle_{\mathcal{S}})\big)
  =(2\pi)\int_{\Omega}\sum_{k\in\mathbb Z}\hat G_k(I)\,\hat\rho_{0,-k}(I)\,
     \lim_{N\to\infty}\frac{1}{N}\sum_{j=1}^{N} 
     \big(e^{-\frac{c^{2}k^{2}}{2}}\;e^{\mathrm{i}\,k\,\omega(I)}\big)^{j}\,dI.
\]
For $k\neq 0$, let $q(I,k):=e^{-\frac{c^{2}k^{2}}{2}}\;e^{\mathrm{i}\,k\,\omega(I)}$; then $|q(I,k)|=e^{-\frac{c^{2}k^{2}}{2}}<1$, and
\[
  \left|\frac{1}{N}\sum_{j=1}^{N}q(I,k)^{\,j}\right|
  \le \frac{1}{N}\frac{|q(I,k)|}{1-|q(I,k)|}\xrightarrow[N\to\infty]{}0.
\]
For $k=0$, the average equals $1$. Hence
\begin{equation}\label{zhishujishoulian}
  \lim_{N\to\infty}\mathbb{E}\big(V_{N}(\langle G \rangle_{\mathcal{S}})\big)
  =(2\pi)\int_{\Omega}\hat G_{0}(I)\,\hat\rho_{0,0}(I)\,dI
  =\int_{\Omega}\bar G(I)\Big(\int_{\mathbb T}\rho_{0}(I,\theta)\,d\theta\Big)\,dI
  =\langle \bar G(I)\rangle_{0}.
\end{equation}
This completes the proof.
\end{proof}

\begin{remark}\label{remark3.2}
  It is worth noting that Equation \eqref{zhishujishoulian} indicates that, after incorporating Brownian motion into the iterative map, the convergence rate of the statistical ensemble becomes exponential and is influenced by the strength of the Brownian motion, represented by the constant $c$ and we will use an actual numerical example in the Section \ref{section5} to verify this point.
\end{remark}
It is worth noting that, in the above theorem, no assumptions analogous to those in Theorem \ref{theorem2.1} are required for the iteration map. The convergence rate of the statistical ensemble is of order $N^{-1}$, and the final state of the statistical ensemble depends only on the initial distribution and the observation function, but is independent of the map $\mathcal{S}_{H_0}$.
Next, we will extend the theorem to the case of general stochastic processes, which means 
\begin{equation}\label{mapwithX}
  \mathcal{S}_{H_0}^{j}(I,\theta)=(I,\theta+j\omega(I)+cX_{j}),
\end{equation}
where $X_t$ is a stochastic process with some specific conditions. However, it is important to note that the favorable result presented in Theorem \ref{theorem3.1} relies on the unique properties of Brownian motion. For general stochastic processes case, it is necessary to impose additional assumptions on the system $H_0$ itself.


\begin{theorem}\label{theorem3.2}
Let $\omega:\Omega\to\mathbb{R}$ satisfy the non-resonance condition \eqref{nonresonantcondition} and let $G\in \bm{\mathrm{C}}_{b}(\Omega\times\mathbb{T})$.
Assume the stochastic process $(X_j)_{j\ge1}$ is independent of $(I,\theta)$ and, for every $k\in\mathbb{Z}\setminus\{0\}$, the Ces\`aro decay
\begin{equation} \label{H0} 
\frac{1}{N}\sum_{j=1}^{N}\big|a_j^{(k)}\big|\longrightarrow 0\qquad(N\to\infty),
\quad 
a_j^{(k)}:=\mathbb{E}\big(
e^{\mathrm{i}\,c k\,X_j}
\big)
\end{equation}
holds. If the initial density satisfies $\rho_0\in \bm{\mathrm{L}}^{2}(\Omega\times\mathbb{T})$, then
\[
\lim_{N\to\infty}\mathbb{E}\big(V_N(\langle G\rangle_{\mathcal S})\big)
=\langle \bar G(I)\rangle_0,\qquad 
\bar G(I)=\frac{1}{2\pi}\int_{\mathbb T}G(I,\theta)\,d\theta .
\]
\end{theorem}

\begin{proof}
By Parseval (as in Theorem~\ref{theorem2.1}),
\begin{align*}
\mathbb{E}\big(V_{N}(\langle G\rangle_{\mathcal S})\big)
&=(2\pi)\int_{\Omega}\sum_{k\in\mathbb Z}\hat G_k(I)\,\hat\rho_{0,-k}(I)\,S_N(I,k)\,dI,
\end{align*}
where, using independence of $(X_j)$ from $(I,\theta)$,
\[
S_N(I,k):=\frac{1}{N}\sum_{j=1}^{N} e^{\mathrm{i}\,j k\,\omega(I)}\,a_j^{(k)},
\qquad a_j^{(k)}=\mathbb{E}\big(e^{\mathrm{i}\,c k\,X_j}\big).
\]

For $k=0$, since $a_j^{(0)}=1$, we have $S_N(I,0)=1$.
For $k\neq0$, set $\varphi(I,k):=k\,\omega(I)$. Non-resonance gives $\varphi(I,k)\notin 2\pi\mathbb Z$. 
The Dirichlet (Abel) bound yields, for any sequence $(a_j)$,
\begin{equation}\label{eq:Dirichlet}
\left|\frac{1}{N}\sum_{j=1}^{N} e^{\mathrm{i}j\varphi}\,a_j\right|
\le \frac{2}{|e^{\mathrm{i}\varphi}-1|}\cdot \frac{1}{N}\sum_{j=1}^{N}|a_j|.
\end{equation}

Applying \eqref{eq:Dirichlet} with $a_j=a_j^{(k)}$ and using \eqref{H0} gives
$S_N(I,k)\to 0$ for each fixed $k\neq0$ and every $I\in\Omega$.

To justify exchanging $\lim$ with $\sum_k\int_\Omega$, note that characteristic functions satisfy
$|a_j^{(k)}|\le 1$, hence $|S_N(I,k)|\le 1$ for all $I,k,N$. Therefore,


\begin{align*}
\Big|\sum_{k}\hat G_k(I)\,\hat\rho_{0,-k}(I)\,S_N(I,k)\Big|
&\leq \Big(\sum_{k}|\hat G_k(I)|^2\Big)^{1/2}\Big(\sum_{k}|\hat\rho_{0,-k}(I)|^2|S_N(I,k)|^2\Big)^{1/2} \\
&\leq \Big(\sum_{k}|\hat G_k(I)|^2\Big)^{1/2}\Big(\sum_{k}|\hat\rho_{0,-k}(I)|^2\Big)^{1/2}.
\end{align*}

Integrating in $I$ and using Cauchy--Schwarz in $I$ and Plancherel on $\mathbb{T}$, we obtain an
$N$-independent integrable dominant:
\begin{equation*}
(2\pi)\int_{\Omega}\Big(\sum_{k}|\hat G_k(I)|^2\Big)^{1/2}\!\Big(\sum_{k}|\hat\rho_{0,-k}(I)|^2\Big)^{1/2}dI
\leq \|G\|_{L^2(\Omega\times\mathbb T)}\,\|\rho_0\|_{L^2(\Omega\times\mathbb T)}<\infty .
\end{equation*}
Dominated convergence then gives
\begin{equation*}
\lim_{N\to\infty}\mathbb{E}\big(V_{N}(\langle G\rangle_{\mathcal S})\big)
=(2\pi)\int_{\Omega}\sum_{k\in\mathbb Z}\hat G_k(I)\,\hat\rho_{0,-k}(I)\,\lim_{N\to\infty}S_N(I,k)\,dI.
\end{equation*}
The limit equals $1$ for $k=0$ and $0$ for $k\neq0$, hence
\begin{equation*}
(2\pi)\int_{\Omega}\hat G_0(I)\,\hat\rho_{0,0}(I)\,dI
=\int_{\Omega}\bar G(I)\Big(\int_{\mathbb T}\rho_0(I,\theta)\,d\theta\Big)\,dI
=\langle \bar G(I)\rangle_0 .
\end{equation*}
\end{proof}

For condition \eqref{H0}, we provide two relatively weaker alternative conditions in the form of the following two corollaries.

\begin{corollary}[H1: Geometric decay]\label{cor:H1}
If for each $k\neq0$ there exist constants $C_k\ge0$ and $r_k\in(0,1)$ such that
$|a_j^{(k)}|\le C_k\,r_k^{\,j}$ for all $j\ge1$, then condition \eqref{H0} holds. 
Consequently, the conclusion of Theorem~\ref{theorem3.2} follows. 
In fact,
\[
 \big|S_N(I,k)\big|
 \le \frac{2}{|e^{\mathrm{i}k\omega(I)}-1|}\cdot
     \frac{C_k}{N}\cdot\frac{1-r_k^{\,N}}{1-r_k}
 \le \frac{2C_k}{|e^{\mathrm{i}k\omega(I)}-1|\,(1-r_k)}\cdot\frac{1}{N}.
\]
\end{corollary}


\begin{corollary}[H3: Strongly mixing sequence]\label{cor:H3}
Let $(X_j)_{j\ge1}$ be strictly stationary with strong mixing coefficients $\alpha(j)$ and assume
$\mathbb{E}|X_0|^{2+\delta}<\infty$ for some $\delta>0$. Let 
$\varphi(t):=\mathbb{E}(e^{\mathrm{i}tX_0})$ and $a_j^{(k)}:=\mathbb{E}(e^{\mathrm{i}ckX_j})$.
If 
\[
\sum_{j=1}^{\infty}\alpha(j)^{\delta/(2+\delta)}<\infty,
\]
then for every $k\neq0$ we have
\[
\sum_{j=1}^{\infty}\big|a_j^{(k)}-\varphi(ck)\big|<\infty,
\]
and consequently the Cesàro terms in the proof of Theorem~\ref{theorem3.2} satisfy
\[
S_N(I,k):=\frac{1}{N}\sum_{j=1}^N e^{\mathrm{i}\,jk\,\omega(I)}\,a_j^{(k)} \;\longrightarrow\; 0
\quad(N\to\infty)
\]
for all $I\in\Omega$ and all $k\neq0$. Hence the conclusion of Theorem~\ref{theorem3.2} holds.
\end{corollary}

\begin{proof}
By Bradley's covariance inequality (see, e.g., Bradley's lemma for $\alpha$-mixing),
for any $k\in\mathbb Z$,
\[
\big|a_j^{(k)}-\varphi(ck)\big|
=\big|\mathbb{E}\big(e^{\mathrm{i}ckX_j}\big)-\mathbb{E}\big(e^{\mathrm{i}ckX_0}\big)\big|
\le C\,\alpha(j)^{\delta/(2+\delta)},
\]
where $C$ depends on $\delta$ and $\mathbb{E}|X_0|^{2+\delta}$. Hence the series
$\sum_{j\ge1}|a_j^{(k)}-\varphi(ck)|$ converges whenever 
$\sum_{j\ge1}\alpha(j)^{\delta/(2+\delta)}<\infty$.

Decompose, for $k\neq0$ and any $I$,
\[
S_N(I,k)=\varphi(ck)\,\frac{1}{N}\sum_{j=1}^N e^{\mathrm{i}\,jk\,\omega(I)}
+\frac{1}{N}\sum_{j=1}^N e^{\mathrm{i}\,jk\,\omega(I)}\big(a_j^{(k)}-\varphi(ck)\big).
\]
Since $\omega$ is nonresonant, $k\,\omega(I)\notin2\pi\mathbb Z$, and the Dirichlet bound gives
\[
\left|\frac{1}{N}\sum_{j=1}^N e^{\mathrm{i}\,jk\,\omega(I)}\right|
\le \frac{2}{N\,|e^{\mathrm{i}k\omega(I)}-1|} \xrightarrow[N\to\infty]{} 0 .
\]
For the second term, Dirichlet(Abel) summation yields
\[
\left|\frac{1}{N}\sum_{j=1}^N e^{\mathrm{i}\,jk\,\omega(I)}\big(a_j^{(k)}-\varphi(ck)\big)\right|
\le \frac{2}{|e^{\mathrm{i}k\omega(I)}-1|}\cdot
\frac{1}{N}\sum_{j=1}^N \big|a_j^{(k)}-\varphi(ck)\big|
\longrightarrow 0 ,
\]
because the series of $\big|a_j^{(k)}-\varphi(ck)\big|$ is absolutely summable. Hence
$S_N(I,k)\to0$ for each $k\neq0$, which is exactly the decay needed in the proof of
Theorem~\ref{theorem3.2}. 
\end{proof}

\begin{remark}[On the deterministic counterexample]\label{rem:counterexample}
Consider the deterministic choice
\[
  X_j=\frac{j}{c}\,\beta \quad\text{with}\quad \beta=-\,\frac{\omega(I)}{k}
  \qquad (k\neq 0),
\]
or, equivalently, $X_j=\frac{j}{c}\big(2\pi m-\langle k,\omega(I)\rangle\big)$ for some $m\in\mathbb Z$.
Then
\[
  a_j^{(k)}=\mathbb{E}\left(e^{\mathrm{i}\,c k\,X_j}\right)
  =e^{-\mathrm{i}\,j k\,\omega(I)},
\]
so that the oscillatory factor cancels:
\[
  e^{\mathrm{i}\,j k\,\omega(I)}\,a_j^{(k)} \equiv 1 .
\]
Hence the summand in the Ces\`aro average equals
$\hat G_k(I)\,\hat\rho_{0,-k}(I)$ and the Ces\`aro mean does \emph{not} vanish.
This shows that  $|a_j^{(k)}|\le C$  is insufficient.

Under our hypotheses in Theorem \ref{theorem3.2}, Corollary \ref{cor:H1} and Corollary \ref{cor:H3}, this pathology is excluded. 
Indeed, 
\begin{equation*}
\forall\,k\neq 0:\qquad \frac{1}{N}\sum_{j=1}^{N}\big|a_j^{(k)}\big|\;\longrightarrow\;0
\quad (N\to\infty)
\end{equation*}
fails in the above deterministic construction (since $|a_j^{(k)}|=1$ for all $j$), 
and therefore such cases are ruled out. 
Thus the counterexample above cannot occur under the new assumptions.
\end{remark}

Based on the results, the convergence of statistical ensembles resembles the Law of Large Numbers, which naturally raises the question of whether a Central Limit Theorem holds. 

\section{The CLT of Statistical Ensemble}\label{sec:CLT}
Throughout this section we work in the one-dimensional angle case:
\[
  \theta\in\mathbb T,\qquad \omega:\Omega\subset\mathbb R\to\mathbb R,
  \qquad \mathcal F^{\,j}(I,\theta)=(I,\theta+j\,\omega(I)).
\]
We keep the Fourier normalization and Parseval convention used earlier in the paper.
To simplify notation, set
\begin{equation}\label{4111}
  \begin{aligned}
    X_{j}&= G(\mathcal{F}^{j}(I,\theta)) - \langle \bar{G}(I) \rangle_{0},\\
    \mathcal{X}_{N}&= \frac{1}{\sqrt{N}}\sum_{j=1}^{N}X_{j}.
  \end{aligned}
\end{equation}
(The symbol $X_j$ is used only in this section with the above meaning.)

\medskip
For an integrable Hamiltonian system, the following result was obtained in~\cite{444555}; we restate it in our one-dimensional setting:

\begin{theorem}\label{yiyoudingli1}
In system \eqref{integrablehamiltoniansystem}, suppose $\omega \in C^2(\Omega,\mathbb R)$ has no critical points. Assume the initial condition is given by a probability density $\rho_0 \in \bm{\mathrm{L}}^{1}(\Omega \times \mathbb{T})$. Then for any $G \in C_{b}(\Omega \times \mathbb{T})$,
\[
\lim_{t \to \infty} \langle G \rangle_t = \langle \bar{G}(I) \rangle_0.
\]
\end{theorem}

The above concerns the one-parameter flow $\varphi_{t}(I,\theta)=(I,\theta+\omega(I)t)$ and the same conclusion holds for the discrete map $\mathcal{F}^{j}(I,\theta)=(I,\theta+\omega(I)j)$. If $G(\mathcal F^j(I,\theta))$ is viewed as a random variable with law induced by the initial density $\rho_0$, then $\langle G \rangle_j=\mathbb E(G(\mathcal F^j))$, and
\begin{equation}\label{xizongshoulian}
  \lim_{j \to \infty}\mathbb E\left(G(\mathcal{F}^{j}(I,\theta))\right)= \langle \bar{G}(I) \rangle_{0}.
\end{equation}
Consequently (apply \eqref{xizongshoulian} with $G^2$),
\[
  \operatorname{Var}\left(G(\mathcal{F}^{j}(I,\theta))\right)
  = \mathbb E\left(G^2(\mathcal{F}^{j}(I,\theta))\right)
    -\Big(\mathbb E\left(G(\mathcal{F}^{j}(I,\theta))\right)\Big)^2
  \longrightarrow \langle \bar{G^2}(I) \rangle_{0} - \langle \bar{G}(I) \rangle_{0}^{2}.
\]
In the following we write
\[
  \sigma^2 := \langle \bar{G^2}(I) \rangle_{0} - \langle \bar{G}(I) \rangle_{0}^{2}.
\]

\begin{theorem}
  Under the assumptions of Theorem \ref{yiyoudingli1}, we have
  \begin{equation}\label{4.2.1}
    \lim_{j \to \infty} \mathbb E(X_j) = 0, 
    \qquad 
    \lim_{j \to \infty} \operatorname{Var}(X_j) =\sigma^2.
  \end{equation}
\end{theorem}

\begin{proof}
By definition,
\[
\mathbb E(X_j)=\mathbb E\left(G(\mathcal F^j)\right)-\langle \bar G\rangle_0\to \langle \bar G\rangle_0-\langle \bar G\rangle_0=0
\]
by \eqref{xizongshoulian}. Replacing $G$ by $G^2$ in \eqref{xizongshoulian} yields
\[
\mathbb E\left(G^2(\mathcal F^j)\right)\to \langle \bar{G^2}(I)\rangle_0.
\]
Therefore
\[
\operatorname{Var}(X_j)=\operatorname{Var}\left(G(\mathcal F^j)-\langle \bar G\rangle_0\right)
=\mathbb E\left(G^2(\mathcal F^j)\right)-\left(\mathbb E\left(G(\mathcal F^j)\right)\right)^2
\to \langle \bar{G^2}\rangle_0-\langle \bar G\rangle_0^2=\sigma^2.
\]
\end{proof}

\subsection*{Auxiliary lemmas and their proofs}

\begin{lemma}[Exponential covariance decay]\label{Hcov}
Assume:
\begin{itemize}
\item \textbf{(Analytic frequency)} $\omega$ admits a holomorphic extension to the strip 
$\{I+\mathrm{i}y: |y|\le \eta\}$ and there exists $\gamma_{*}>0$ such that
\[
  \mathrm{sgn}(y)\,\mathrm{Im}\,\omega(I+\mathrm{i}y)\ \ge\ \gamma_{*}\,|y|,
  \qquad 0<|y|\le \eta,
\]
uniformly in $I\in\Omega$.
\item \textbf{(Analytic observables in $I$ and $\theta$)} $G,\rho_{0}$ are real-analytic in $\theta$ and, for their Fourier series
\[
  G(I,\theta)=\bar G(I)+\sum_{m\ne 0}\hat G_{m}(I)\,e^{\mathrm{i}m\theta},\quad
  \rho_{0}(I,\theta)=\sum_{q\in\mathbb Z}\hat\rho_{0,q}(I)\,e^{\mathrm{i}q\theta},
\]
there exist $\alpha>0$, $\eta>0$ and $C>0$ such that for all $|y|\le\eta$ and all $I$,
\[
  \sup_{|y|\le\eta}\,|\hat G_{m}(I+\mathrm{i}y)|\le C e^{-\alpha |m|},\qquad 
  \sup_{|y|\le\eta}\,|\hat\rho_{0,q}(I+\mathrm{i}y)|\le C e^{-\alpha |q|}.
\]
(Equivalently: $\hat G_m$ and $\hat\rho_{0,q}$ admit holomorphic extensions in $I$ to the same strip with uniform exponential decay in $|m|,|q|$.)
\end{itemize}
Then there exist constants $C_1,\eta'>0$ such that, for all $j\ge1$,
\[
  \left|\operatorname{Cov}(X_0,X_{j})\right|\ \le\ C_1\,e^{-\eta' j}.
\]
Moreover, for all $k\ge1$,
\[
  \sup_{j\ge1}\left|\operatorname{Cov}(X_j,X_{j+k})\right|\ \le\ C_1\,e^{-\eta' k},
\]
and consequently
\[
  \sum_{k=1}^{\infty}\sup_{j\ge1}\left|\operatorname{Cov}(X_j,X_{j+k})\right|<\infty.
\]
\end{lemma}

\begin{proof}
Using the Fourier expansions and integrating over $\theta$,
\[
\mathbb E\left(G(\mathcal F^0)G(\mathcal F^j)\right)
=2\pi\sum_{m,n\in\mathbb Z}\int_{\Omega}\hat G_m(I)\hat G_n(I)\hat\rho_{0,-(m+n)}(I)\,
e^{\mathrm{i}n j \omega(I)}\,dI.
\]
By the analytic assumptions, the integrand is holomorphic and uniformly bounded in the strip $\{I+\mathrm i y:|y|\le\eta\}$. Shift the contour to $I+\mathrm{i}y\,\mathrm{sgn}(n)$ with any $y\in(0,\eta]$. On the shifted contour,
\[
\left|e^{\mathrm{i}n j\omega(I+\mathrm{i}y\,\mathrm{sgn}(n))}\right|
=\exp\left(-j\,|n|\,\mathrm{Im}\,\omega(I+\mathrm{i}y\,\mathrm{sgn}(n))\right)
\le \exp\left(-j\,|n|\,\gamma_*\,y\right).
\]
Summing in $m,n$ using the exponential bounds on the coefficients yields
\[
\left|\mathbb E\left(G(\mathcal F^0)G(\mathcal F^j)\right)\right|\le C e^{-\eta' j}.
\]
The same argument gives $\big|\mathbb E\left(G(\mathcal F^j)\right)-\langle \bar G\rangle_0\big|\le C e^{-\eta' j}$, so the product-of-means term is uniformly bounded. Hence $\big|\operatorname{Cov}(X_0,X_j)\big|\le C_1 e^{-\eta' j}$.
For $\operatorname{Cov}(X_j,X_{j+k})$ the prefactor $e^{\mathrm{i}(m+n)j\omega(I)}$ has unit modulus, and the same contour shift in $n$ yields a bound independent of $j$:
\[
\sup_{j\ge1}\left|\mathbb E\left(G(\mathcal F^j)G(\mathcal F^{j+k})\right)\right|\le C e^{-\eta' k}.
\]
This gives the claimed estimates and summability.
\end{proof}

\begin{lemma}\label{Hmom}
Let $G\in C_b(\Omega\times\mathbb T)$. Then, for $X_j$ defined by \eqref{4111},
\[
  \sup_{j\ge1}\,\mathbb E|X_j|^{3}\ \le\ (2\|G\|_\infty)^{3}.
\]
\end{lemma}

\begin{proof}
Since $|X_j|\le |G(\mathcal F^{\,j})|+|\langle\bar G\rangle_0|\le 2\|G\|_\infty$ and $\rho_0$ is a probability density,
\[
  \mathbb E|X_j|^{3}=\int_{\Omega\times\mathbb T}|X_j(I,\theta)|^{3}\,\rho_0(I,\theta)\,dI\,d\theta
  \le (2\|G\|_\infty)^{3}.
\]
\end{proof}

\begin{lemma}\label{Hmean}
Assume $\Omega\subset\mathbb R$ is bounded, $\omega\in C^{1}(\Omega)$ with $\inf_{I\in\Omega}|\omega'(I)|=:c_0>0$, and $G,\rho_0$ are real-analytic in $\theta$ with
\[
  G(I,\theta)=\bar G(I)+\sum_{k\ne0}\hat G_k(I)e^{\mathrm{i}k\theta},\quad
  \rho_0(I,\theta)=\sum_{k\in\mathbb Z}\hat\rho_{0,k}(I)e^{\mathrm{i}k\theta},
\]
where, for some $\alpha>0$ and all $I\in\Omega$,
\[
  |\hat G_k(I)|+|\partial_I\hat G_k(I)|+|\hat\rho_{0,k}(I)|+|\partial_I\hat\rho_{0,k}(I)|
  \le C\,e^{-\alpha|k|}.
\]
Then there exists $C>0$ such that for all $j\ge1$,
\[
  \left|\mathbb E(X_j)\right|\ \le\ \frac{C}{j}.
\]
Consequently, $\mathbb E(X_j)\to0$ and $\frac{1}{\sqrt N}\sum_{j=1}^{N}|\mathbb E(X_j)|\to0$ as $N\to\infty$.
\end{lemma}

\begin{proof}
By the Fourier expansion in $\theta$ and $\mathcal F^{\,j}(I,\theta)=(I,\theta+j\omega(I))$,
\[
  \mathbb E(X_j)
  = 2\pi\sum_{k\ne0}\int_{\Omega}\hat G_k(I)\,\hat\rho_{0,-k}(I)\,e^{\mathrm{i}k\,\omega(I)\,j}\,dI.
\]
Set $A_k(I):=\hat G_k(I)\hat\rho_{0,-k}(I)$. Since $A_k,\partial_I A_k$ are uniformly bounded by $C e^{-\alpha|k|}$ and $|\omega'(I)|\ge c_0>0$ on the bounded interval $\Omega$, we integrate by parts:
\[
\int_{\Omega}A_k(I)e^{\mathrm i k j\omega(I)}\,dI
=\left.\frac{A_k(I)}{\mathrm i k j\,\omega'(I)}\,e^{\mathrm i k j\omega(I)}\right|_{\partial\Omega}
-\int_{\Omega}\partial_I\!\left(\frac{A_k(I)}{\mathrm i k j\,\omega'(I)}\right)\,e^{\mathrm i k j\omega(I)}\,dI.
\]
Both the boundary term and the integral are $O\!\big(e^{-\alpha|k|}/(j|k|)\big)$, because
$|A_k|\lesssim e^{-\alpha|k|}$, $|\partial_I A_k|\lesssim e^{-\alpha|k|}$ and
$1/|\omega'|\le c_0^{-1}$. Hence
\[
\left|\int_{\Omega}A_k(I)\,e^{\mathrm{i}k\,\omega(I)\,j}\,dI\right|
\ \le\ \frac{C}{j\,|k|}\,e^{-\alpha|k|}.
\]
Summing over $k\ne0$ yields $|\mathbb E(X_j)|\le C/j$, and therefore $\mathbb E(X_j)\to0$ and
$\frac{1}{\sqrt N}\sum_{j=1}^{N}|\mathbb E(X_j)|\to0$ since $\sum_{j\le N}j^{-1}=O(\log N)$.
\end{proof}


\begin{lemma}[Lindeberg condition]\label{lem:lindeberg}
Under Lemma \ref{Hmom}, the Lindeberg condition holds for the triangular array $Y_{N,j}:=X_j/\sqrt N$: for every $\varepsilon>0$,
\[
 \sum_{j=1}^{N}\mathbb{E}\left[\,Y_{N,j}^{2}\,\mathbf{1}_{\{|Y_{N,j}|>\varepsilon\}}\,\right]\xrightarrow[N\to\infty]{}0 .
\]
\end{lemma}

\begin{proof}
Fix $\varepsilon>0$ and $N\in\mathbb N$. By definition $Y_{N,j}=X_j/\sqrt N$, hence
\[
\mathbb{E}\!\left[\,Y_{N,j}^{2}\,\mathbf{1}_{\{|Y_{N,j}|>\varepsilon\}}\,\right]
=\frac{1}{N}\,\mathbb{E}\!\left[\,X_j^{2}\,\mathbf{1}_{\{|X_j|>\varepsilon\sqrt N\}}\,\right].
\]
For any real $x$ and any $a>0$, we have the elementary bound
\[
x^{2}\,\mathbf{1}_{\{|x|>a\}} \le \frac{|x|^{3}}{a}
\quad\text{since}\quad
\mathbf{1}_{\{|x|>a\}}\le \frac{|x|}{a}.
\]
Applying this with $x=X_j$ and $a=\varepsilon\sqrt N$ gives
\[
\mathbb{E}\!\left[\,X_j^{2}\,\mathbf{1}_{\{|X_j|>\varepsilon\sqrt N\}}\,\right]
\le \frac{1}{\varepsilon\sqrt N}\,\mathbb{E}|X_j|^{3}.
\]
Therefore,
\[
\sum_{j=1}^{N}\mathbb{E}\!\left[\,Y_{N,j}^{2}\,\mathbf{1}_{\{|Y_{N,j}|>\varepsilon\}}\,\right]
\le \frac{1}{N}\sum_{j=1}^{N}\frac{1}{\varepsilon\sqrt N}\,\mathbb{E}|X_j|^{3}
=\frac{1}{\varepsilon\sqrt N}\,\Big(\frac{1}{N}\sum_{j=1}^{N}\mathbb{E}|X_j|^{3}\Big).
\]
By Lemma \ref{Hmom} there is a constant $M_3=(2\|G\|_\infty)^3$ such that $\sup_{j\ge1}\mathbb{E}|X_j|^{3}\le M_3$. Hence
\[
\sum_{j=1}^{N}\mathbb{E}\!\left[\,Y_{N,j}^{2}\,\mathbf{1}_{\{|Y_{N,j}|>\varepsilon\}}\,\right]
\le \frac{M_3}{\varepsilon\sqrt N}\xrightarrow[N\to\infty]{}0,
\]
which is exactly the Lindeberg condition.
\end{proof}

\begin{lemma}[Log-characteristic expansion]\label{lem:log-expansion}
Let $S_N:=\sum_{j=1}^{N}(X_j-\mathbb E X_j)$ and $Z_N:=u S_N/\sqrt N$ with $u\in\mathbb R$. Then
\begin{equation}\label{eq:log-phi}
  \log\mathbb E\!\left(e^{\mathrm{i} Z_N}\right)
  = -\,\frac{u^{2}}{2N}\,\mathbb E(S_N^{2}) \;+\; R_N(u),
\end{equation}
where, for all $u\in\mathbb R$, the remainder satisfies
\begin{equation}\label{eq:R-bound}
  |R_N(u)| \ \le\ 
  \frac{|u|^{3}}{6\,N^{3/2}}\,\mathbb E|S_N|^{3}
  \;+\; \frac{u^{4}}{N^{2}}\big(\mathbb E S_N^{2}\big)^{2}
  \;+\; \frac{|u|^{6}}{9\,N^{3}}\big(\mathbb E|S_N|^{3}\big)^{2}.
\end{equation}
In particular, under Lemma \ref{Hcov} we have $\mathbb E S_N^{2}\le C N$ and 
$\mathbb E|S_N|^{3}\le C N^{3/2}$ for some $C>0$, hence for $u$ in compact sets
\begin{equation}\label{eq:R-vanish}
  R_N(u)\xrightarrow[N\to\infty]{}0 .
\end{equation}
\end{lemma}

\begin{proof}
Set $Y_j:=X_j-\mathbb E X_j$, so that $S_N=\sum_{j=1}^N Y_j$ and $\mathbb E Y_j=0$. For every real $z$,
\begin{equation}\label{eq:taylor-exp}
e^{\mathrm{i}z}=1+\mathrm{i}z-\frac{z^{2}}{2}+r(z),
\qquad
\text{with}\quad |r(z)|\le \frac{|z|^{3}}{6}.
\end{equation}
Applying \eqref{eq:taylor-exp} with $z=Z_N:=uS_N/\sqrt N$ and taking expectations yields
\begin{equation}\label{eq:char-expansion}
\mathbb E\!\left(e^{\mathrm{i}Z_N}\right)
=1-\frac{1}{2}\,\mathbb E(Z_N^{2})+\mathbb E\big(r(Z_N)\big),
\end{equation}
because $\mathbb E(Z_N)=\frac{u}{\sqrt N}\sum_{j=1}^{N}\mathbb E Y_j=0$. Since $\mathbb E(Z_N^{2})=\frac{u^{2}}{N}\mathbb E(S_N^{2})$, we can write
\[
\mathbb E\!\left(e^{\mathrm{i}Z_N}\right)
=1 - \frac{u^{2}}{2N}\,\mathbb E(S_N^{2}) + \mathbb E\big(r(Z_N)\big).
\]
Let
\[
w_N:= - \frac{u^{2}}{2N}\,\mathbb E(S_N^{2}) + \mathbb E\big(r(Z_N)\big).
\]
Then $\mathbb E(e^{\mathrm{i}Z_N})=1+w_N$ and
\begin{equation}\label{eq:log-split}
\log\mathbb E\!\left(e^{\mathrm{i}Z_N}\right)
=\log(1+w_N) = w_N + \Delta_N,
\end{equation}
where, for $|w|\le \tfrac12$, the inequality $|\log(1+w)-w|\le 2|w|^{2}$ holds. We next bound $w_N$.

First, by \eqref{eq:taylor-exp} and Jensen,
\begin{equation}\label{eq:rem-bound}
\big|\mathbb E(r(Z_N))\big|\le \mathbb E|r(Z_N)|
\le \frac{1}{6}\,\mathbb E|Z_N|^{3}
= \frac{|u|^{3}}{6\,N^{3/2}}\,\mathbb E|S_N|^{3}.
\end{equation}
Second, $\mathbb E(S_N^{2})\ge 0$, hence
\begin{equation}\label{eq:wNbound}
|w_N|\le \frac{u^{2}}{2N}\,\mathbb E(S_N^{2})
+ \frac{|u|^{3}}{6\,N^{3/2}}\,\mathbb E|S_N|^{3}.
\end{equation}
Under Lemma \ref{Hcov}, the covariance series is absolutely summable:
\[
\sum_{h=0}^{\infty}\sup_{j\ge1}\big|\mathrm{Cov}(Y_j,Y_{j+h})\big|<\infty.
\]
Therefore
\begin{equation}\label{eq:var-linear}
\mathbb E(S_N^{2})
=\sum_{j=1}^{N}\mathbb E(Y_j^{2})
+2\sum_{h=1}^{N-1}\sum_{j=1}^{N-h}\mathrm{Cov}(Y_j,Y_{j+h})
\le C_2\,N
\end{equation}
for some constant $C_2>0$ independent of $N$. Similarly, there exists $C_3>0$ such that
\begin{equation}\label{eq:third-moment}
\mathbb E|S_N|^{3}\le C_3\,N^{3/2}.
\end{equation}
A detailed proof of \eqref{eq:third-moment} can be obtained from standard moment inequalities for weakly dependent sequences with summable covariances
. In our setting, bounded third moments (Lemma \ref{Hmom}) and exponential covariance decay (Lemma \ref{Hcov}) imply such a bound.

Combining \eqref{eq:wNbound}-\eqref{eq:third-moment} we see that $|w_N|\le C(u)/\sqrt N$ for $u$ in compact sets, hence $|w_N|\le \tfrac12$ for all large $N$. Returning to \eqref{eq:log-split} and using $|\Delta_N|\le 2|w_N|^{2}$ together with \eqref{eq:rem-bound} and \eqref{eq:var-linear}, we obtain
\[
\big|\Delta_N\big|
\le 2\left(\frac{u^{2}}{2N}\,\mathbb E(S_N^{2})
+ \frac{|u|^{3}}{6\,N^{3/2}}\,\mathbb E|S_N|^{3}\right)^{2}
\le \frac{2u^{4}}{N^{2}}\big(\mathbb E S_N^{2}\big)^{2}
+ \frac{u^{3}}{3N^{3/2}}\big(\mathbb E S_N^{2}\big)\,\mathbb E|S_N|^{3}
+ \frac{|u|^{6}}{18N^{3}}\big(\mathbb E|S_N|^{3}\big)^{2}.
\]
Absorbing the mixed and last terms into the displayed bound in \eqref{eq:R-bound} (after enlarging the constant if necessary) yields \eqref{eq:R-bound}, and then \eqref{eq:R-vanish} follows from \eqref{eq:var-linear}–\eqref{eq:third-moment}.
Finally, substituting \eqref{eq:char-expansion} into \eqref{eq:log-split} gives
\[
\log\mathbb E\!\left(e^{\mathrm{i}Z_N}\right)
= -\,\frac{u^{2}}{2N}\,\mathbb E(S_N^{2}) + R_N(u)
\]
with $R_N(u)=\mathbb E\big(r(Z_N)\big)+\Delta_N$, which is exactly \eqref{eq:log-phi}–\eqref{eq:R-bound}.
\end{proof}

\begin{lemma}[Variance limit]\label{lem:variance}
Assume Lemmas \ref{Hcov}, \ref{Hmean}, and \ref{Hmom}. Define
\[
  A_{N,k}:=\frac{1}{N-k}\sum_{j=1}^{N-k}\operatorname{Cov}(X_j,X_{j+k}),\qquad 1\le k\le N-1.
\]
Then $\sup_{N\ge k} |A_{N,k}|\le C e^{-\eta' k}$ and
\begin{equation}\label{eq:VN-decomp}
  \frac{1}{N}\,\operatorname{Var}\left(\sum_{j=1}^{N}X_j\right)
= \frac{1}{N}\sum_{j=1}^N \operatorname{Var}(X_j)
  + 2\sum_{k=1}^{N-1}\left(1-\frac{k}{N}\right) A_{N,k}.
\end{equation}
Moreover, if for each $k\ge1$ the Cesàro limit $c_k:=\lim_{N\to\infty}A_{N,k}$ exists, then the limit
\[
  \sigma_{*}^{2}:=\lim_{N\to\infty}\frac{1}{N}\,\operatorname{Var}\left(\sum_{j=1}^{N}X_j\right)
\]
exists, is finite, and satisfies
\[
  \sigma_*^{2}=\sigma^{2}+2\sum_{k=1}^{\infty} c_k,
\]
with absolute convergence of the series.
\end{lemma}

\begin{proof}
Expanding the variance and grouping by lags $k=j-i$ gives \eqref{eq:VN-decomp}. Lemma \ref{Hcov} yields $|A_{N,k}|\le C e^{-\eta' k}$. By \eqref{4.2.1}, $\frac{1}{N}\sum_{j=1}^{N}\operatorname{Var}(X_j)\to \sigma^2$. If $A_{N,k}\to c_k$ for each $k$, dominated convergence with the summable bound $C e^{-\eta' k}$ implies
\[
\sum_{k=1}^{N-1}\left(1-\frac{k}{N}\right) A_{N,k}\ \longrightarrow\ \sum_{k=1}^{\infty} c_k,
\]
hence the stated limit for $\sigma_*^2$.
\end{proof}

\begin{remark}[Existence of $c_k$ under the analytic-strip assumptions]
Under Lemma~\ref{Hcov}'s hypotheses, the limit $\lim_{j\to\infty}\operatorname{Cov}(X_j,X_{j+k})$ exists for each fixed $k$ by the same contour-shift argument: in the Fourier expansion of $$\mathbb E\big(G(\mathcal F^j)G(\mathcal F^{j+k})\big),$$ all terms with $m+n\neq 0$ vanish as $j\to\infty$, leaving only the diagonal part $m+n=0$. Consequently, the Cesàro limit $c_k=\lim_{N\to\infty}A_{N,k}$ exists.
\end{remark}

\subsection*{Main theorem and corollary}

\begin{theorem}\label{thm:CLT-IL}
Assume Lemmas \ref{Hcov}--\ref{lem:variance}, and suppose the Cesàro limits $c_k=\lim_{N\to\infty}A_{N,k}$ exist for all $k\ge1$. Then, with
\[
   \mathcal X_N:=\frac{1}{\sqrt N}\sum_{j=1}^{N}X_j,\qquad
   \sigma_{*}^{2}:=\sigma^{2}+2\sum_{k=1}^{\infty} c_k,
\]
we have the convergence in distribution
\[
 \mathcal X_N \ \Rightarrow\ \mathcal N(0,\sigma_{*}^{2}).
\]
\end{theorem}

\begin{proof}
Let $m_j:=\mathbb E(X_j)$ and $Y_j:=X_j-m_j$. By Lemma~\ref{Hmean},
\(
 \frac{1}{\sqrt N}\sum_{j=1}^{N}m_j\to0.
\)
Hence it suffices to prove a CLT for $S_N:=\sum_{j=1}^{N}Y_j$. Define $Z_N:=uS_N/\sqrt N$. By Lemma~\ref{lem:log-expansion},
\[
 \log \mathbb E\left(e^{\mathrm{i}Z_N}\right)
 = -\,\frac{u^2}{2}\,\frac{\mathbb E(S_N^2)}{N}\;+\;R_N(u),
\]
with $R_N(u)\to0$ for each fixed $u$. By Lemma~\ref{lem:variance} and the existence of the $c_k$,
\(
 \mathbb E(S_N^2)/N \to \sigma_*^2.
\)
Therefore
\(
 \log \phi_N(u)\to -\tfrac12 u^{2}\sigma_*^{2}
\),
and L\'{e}vy's continuity theorem yields $S_N/\sqrt N\Rightarrow \mathcal N(0,\sigma_*^{2})$. Adding back the negligible centering term completes the proof.
\end{proof}

\begin{corollary}
If the deterministic map $\mathcal F$ is replaced by $\mathcal S$ in \eqref{mapwithB} (Brownian perturbation in the angle), then the assumptions of Lemmas \ref{Hcov}--\ref{lem:variance} are satisfied and Theorem \ref{thm:CLT-IL} remains valid. In particular, the convergence of the means is exponentially fast due to the factor $\mathbb E\big(e^{\mathrm{i}ckB_j}\big)=e^{-(c^2 k^2/2)j}$ on the $k$th Fourier mode.

Here, the \emph{lag-$h$ covariance} means
\[
\operatorname{Cov}(X_j,X_{j+h})=\mathbb E\!\left[(X_j-\mathbb E X_j)(X_{j+h}-\mathbb E X_{j+h})\right].
\]
In our notation this is averaged in $j$ as
\[
A_{N,h}:=\frac{1}{N-h}\sum_{j=1}^{N-h}\operatorname{Cov}(X_j,X_{j+h}),\qquad
c_h:=\lim_{N\to\infty}A_{N,h}.
\]
Under the Brownian perturbation, $B_{j+h}-B_j$ is independent of $(I,\theta)$ and Gaussian with variance $h$, hence for each Fourier mode $k\neq 0$,
\[
\mathbb E\!\left(e^{\mathrm{i}\,ck\,(B_{j+h}-B_j)}\right)=e^{-(c^2 k^2/2)h},
\]
so $\operatorname{Cov}(X_j,X_{j+h})$ acquires the multiplicative factor $e^{-(c^2 k^2/2)h}$ modewise. 
Summing over $k\neq 0$ and using the analyticity bounds from Lemma~\ref{Hcov} yields
$|A_{N,h}|\le C\,e^{-\beta h}$ for some $C,\beta>0$, uniformly in $N$. Hence the Cesàro
limits $c_h=\lim_{N\to\infty}A_{N,h}$ exist and satisfy $|c_h|\le C\,e^{-\beta h}$, so that
$\sum_{h\ge1}|c_h|<\infty$. Consequently,
\[
  \sigma_*^2=\sigma^2+2\sum_{h=1}^{\infty}c_h,
\]
with an exponentially fast decorrelation in $h$ induced by the Brownian perturbation.
(The identity $\sigma_*^2=\sigma^2$ would require $c_h\equiv 0$, which does not hold
in general.)
\end{corollary}

\section{A numerical example}\label{section5}

\begingroup
\setkeys{Gin}{width=\linewidth} 

In this section, we present practical numerical simulations. On the one hand, they serve to validate the theoretical results of this paper; on the other hand, they help to illustrate these results more intuitively.

\subsection{Numerical simulation of Brownian motion with random perturbation}
Let $\Omega = (0,+\infty)$, and the Hamiltonian system considered in this section is 
\begin{equation}\label{moni1}
  H(I,\theta)= \alpha I +\beta I^2+\gamma I^{3}, 
\end{equation}
    where $(I,\theta) \in \Omega \times \mathbb{T}$, $\alpha, \ \beta, \ \gamma $ are positive constants. For $(q,p)\in \mathbb{R}^2,$ the canonical transformation is given by:
    \begin{equation*}
      q+\mathrm{i}p=\sqrt{2I}e^{-\mathrm{i}\theta}.
    \end{equation*}
    The probability density function describing the initial value is
    \begin{equation}\label{moni2}
      \rho_{0}(I,\theta)=\frac{1}{2\pi \varepsilon_{0}}e^{-I/\varepsilon_0}e^{-q_{0}^{2}/2\varepsilon_{0}}\exp(\frac{q_{0}}{\varepsilon_{0}}\sqrt{2I}\cos \theta).
    \end{equation}
The observable function is chosen as 
\begin{equation*}
  G(I,\theta)=\sqrt{2I}e^{-\mathrm{i}\theta},
\end{equation*}
which allows us to consider it as an observation of the position of $(q, p)$ in phase space. And one can easily get $\bar{G}=0=\langle \bar{G} \rangle_0$.
    By setting the parameters in equations \eqref{moni1} and \eqref{moni2} to $\alpha = 0.3$, $\beta = 0.1$, $\gamma = 0.005$, $\varepsilon_0 = 0.01$, $q_0=1.0$ and $p_0=0$ with 10,0000 initial points and $t$ iterations, we obtain the following results.

\begin{figure}[htbp]
  \centering
  \includegraphics[width=0.48\textwidth]{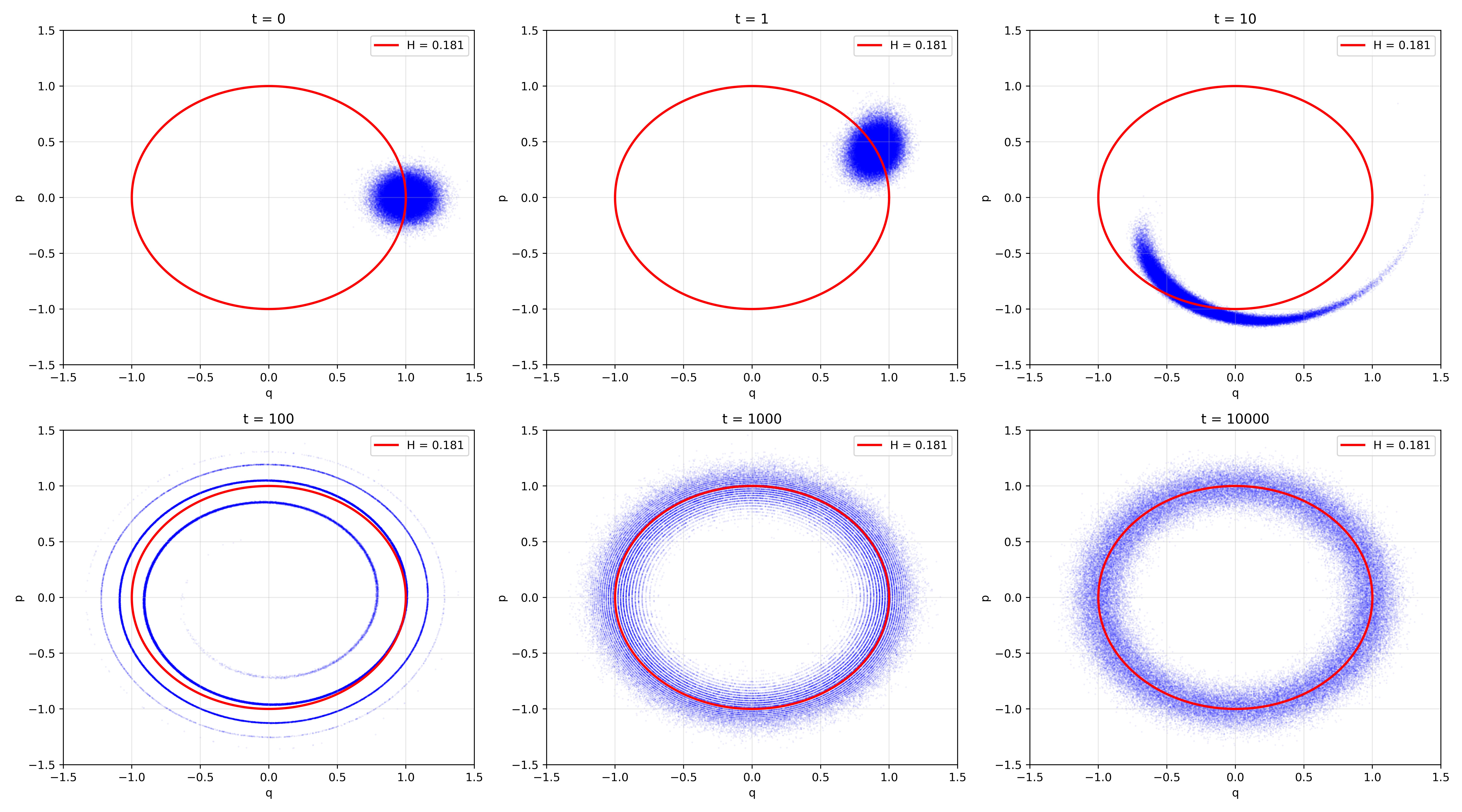}
  \caption{Phase space evolution of 10,0000 initial points under iterative map $\mathcal{F}$ (see~\ref{diedaiyingshequeding}) induced by discrete integrable Hamiltonian systems at different time steps \( t = 0, 1, 10, 100, 1000, 10000 \). The red curve denotes the energy level set \( H = 0.181 \), and the blue dots represent the ensemble distribution.}
  \label{fig:hamilton-phase-evolution}
\end{figure}
   
\autoref{fig:hamilton-phase-evolution} and \autoref{fig:phase-space-zoom} provide a good validation of Theorem \ref{theorem2.1}. From \autoref{fig:hamilton-phase-evolution}, it can be observed that the points in phase space tend to reach equilibrium over a long time horizon, with particular attention drawn to the fifth and sixth subfigures. In the fifth subfigure, distinct stripe patterns—referred to here as \emph{ensemble ripples}~are visible. The sixth further illustrates the overall trend toward equilibrium. These two figures serve as clear references for comparison with the following perturbed cases. To improve clarity for the reader, enlarged versions of these two subfigures are presented in \autoref{fig:phase-space-zoom}.

\begin{figure}[htbp]
  \centering
    \includegraphics[width=0.48\textwidth]{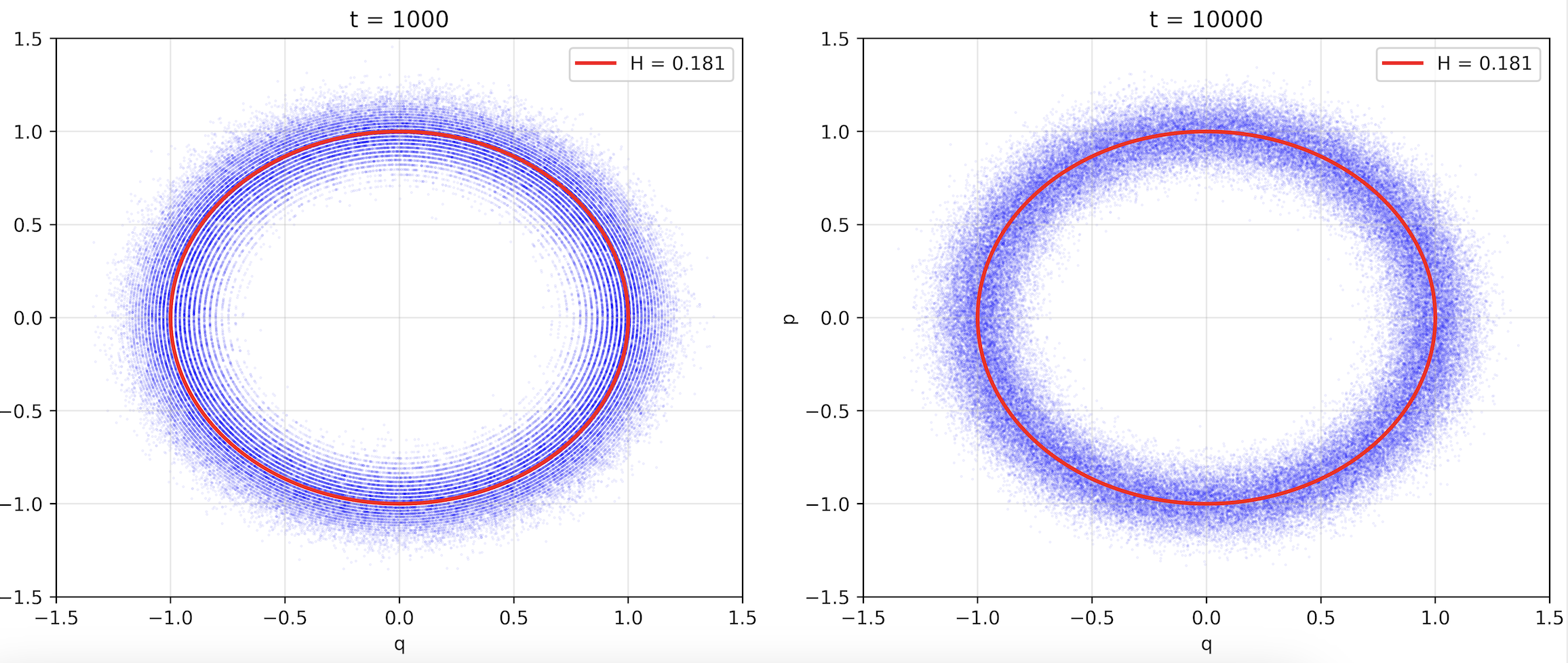}
  \caption{Zoomed view of phase space at $t = 1000$ and $t = 10000$.}
  \label{fig:phase-space-zoom}
\end{figure}

\autoref{fig:hamilton-noise-comparison} illustrates the evolution of $(q, p)$ in phase space under the iterative map $\mathcal{S}$ \eqref{mapwithB}, in the presence of Brownian perturbations.

\begin{figure}[!htbp]
  \centering
  \begin{minipage}[t]{0.48\textwidth}
    \centering
    \includegraphics[width=\textwidth]{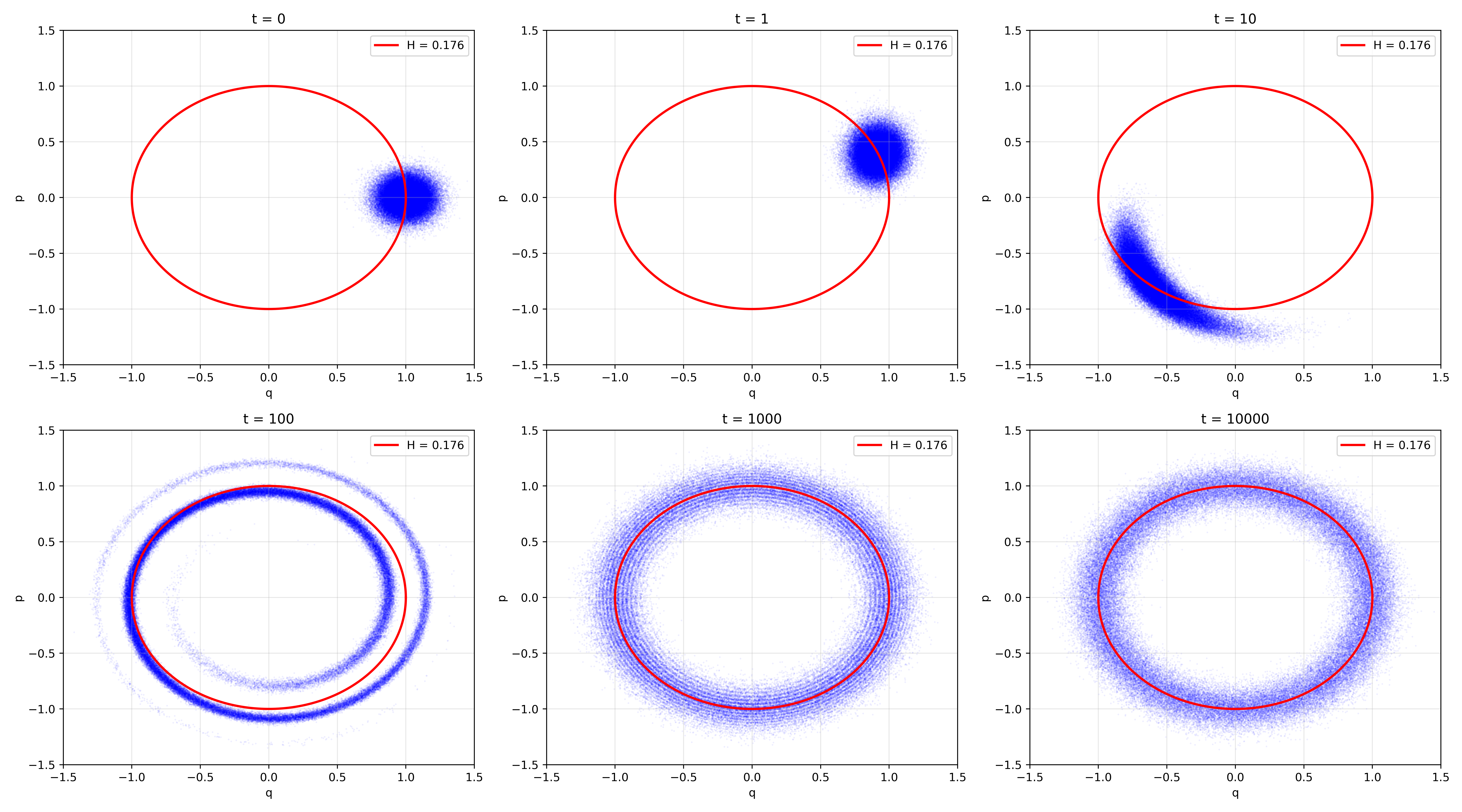}
    \subcaption{Perturbation with intensity $c=0.05$}
  \end{minipage}
  \hfill
  \begin{minipage}[t]{0.48\textwidth}
    \centering
    \includegraphics[width=\textwidth]{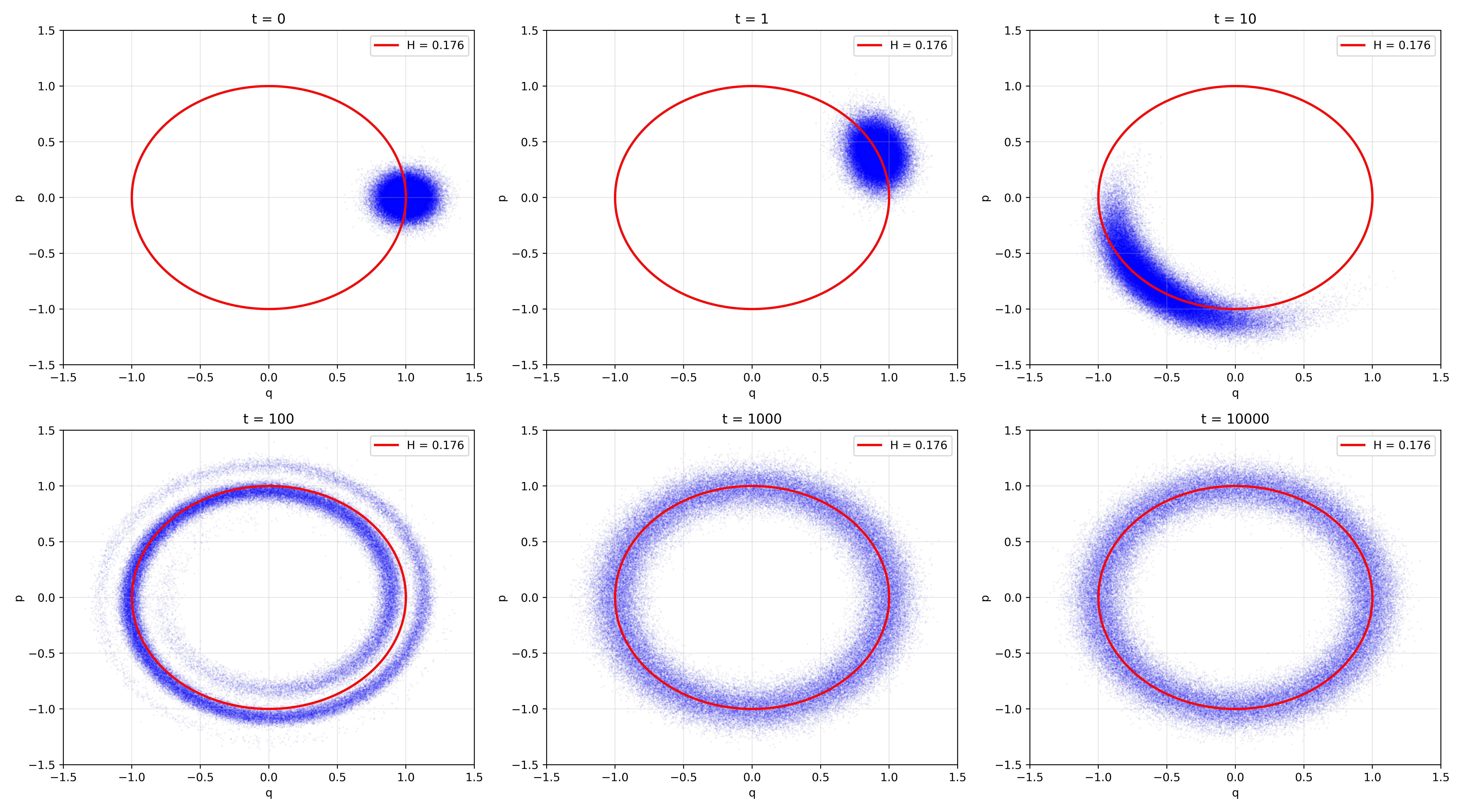}
    \subcaption{Perturbation with intensity $c=0.1$}
  \end{minipage}
  \caption{Phase space evolution under stochastic perturbations $B_t$ of different intensities $c$. The ensemble becomes increasingly diffused with larger noise intensity.}
  \label{fig:hamilton-noise-comparison}
\end{figure}

\begin{figure}[!htbp]\ContinuedFloat
  \centering
  \begin{minipage}[t]{0.48\textwidth}
    \centering
    \includegraphics[width=\textwidth]{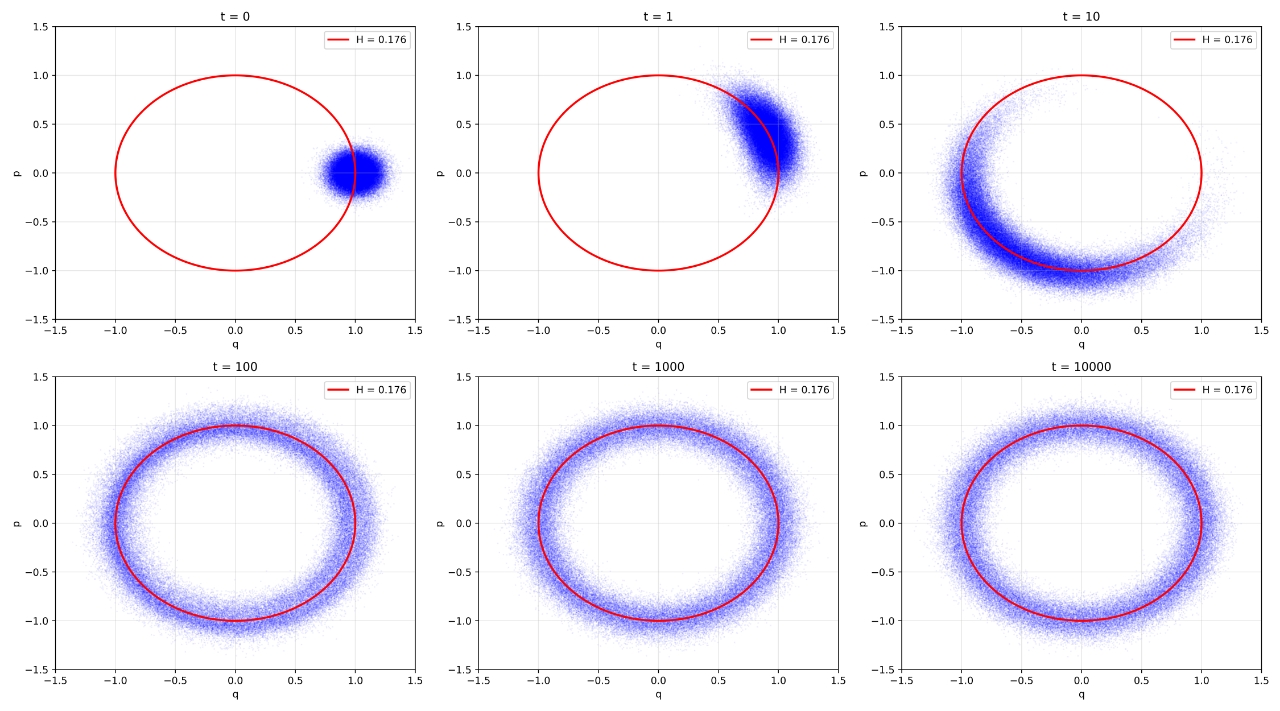}
    \subcaption{Perturbation with intensity $c=0.2$}
  \end{minipage}
  \caption[]{Phase space evolution under stochastic perturbations $B_t$ of different intensities $c$ (continued).}
\end{figure}

\FloatBarrier  

\begin{figure}[htbp]
  \centering
  \begin{minipage}[t]{0.48\textwidth}
    \centering
    \includegraphics[width=\textwidth]{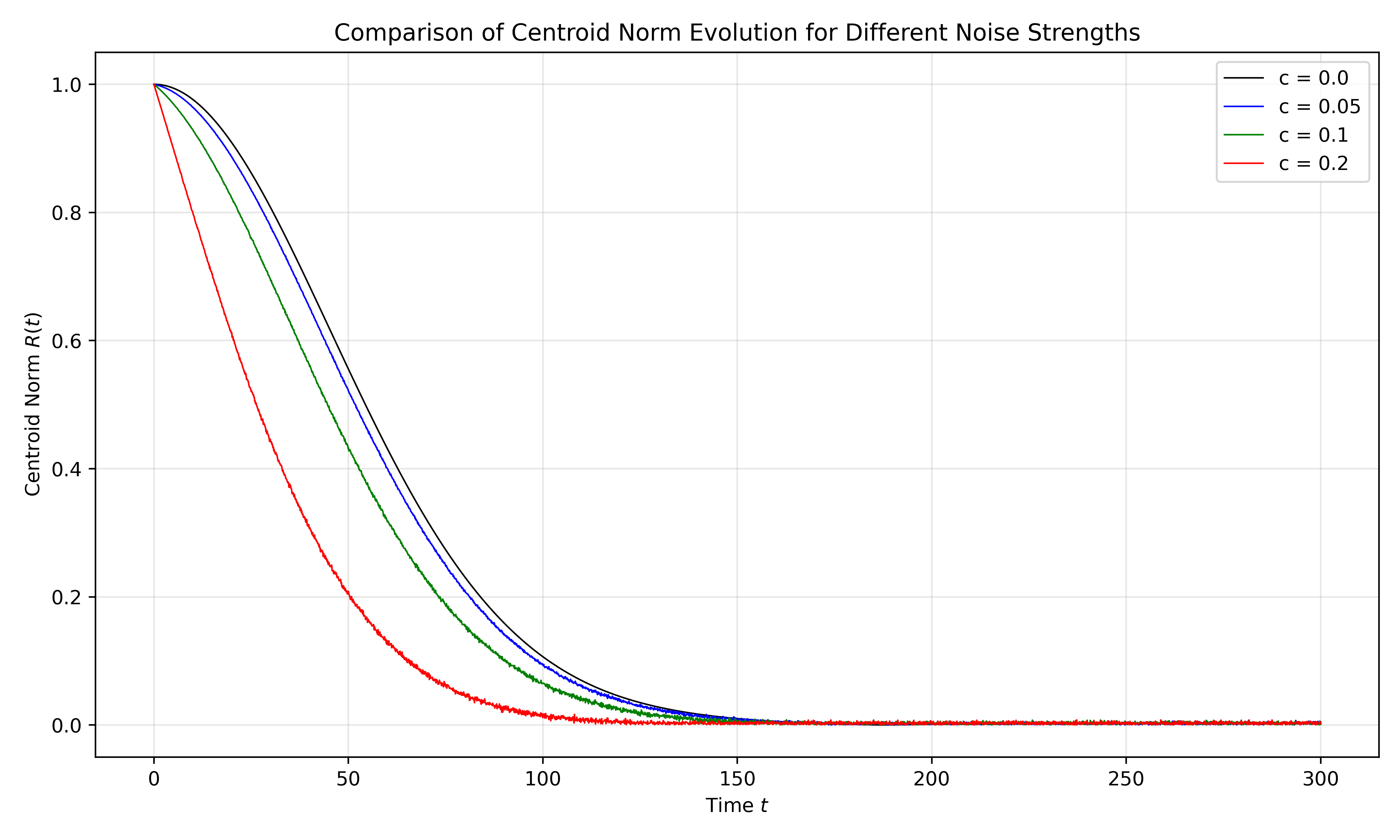}
    \subcaption{Centroid norm evolution for different noise intensities}
    \label{fig:centroid-norm}
  \end{minipage}
  \hfill
  \begin{minipage}[t]{0.48\textwidth}
    \centering
    \includegraphics[width=\textwidth]{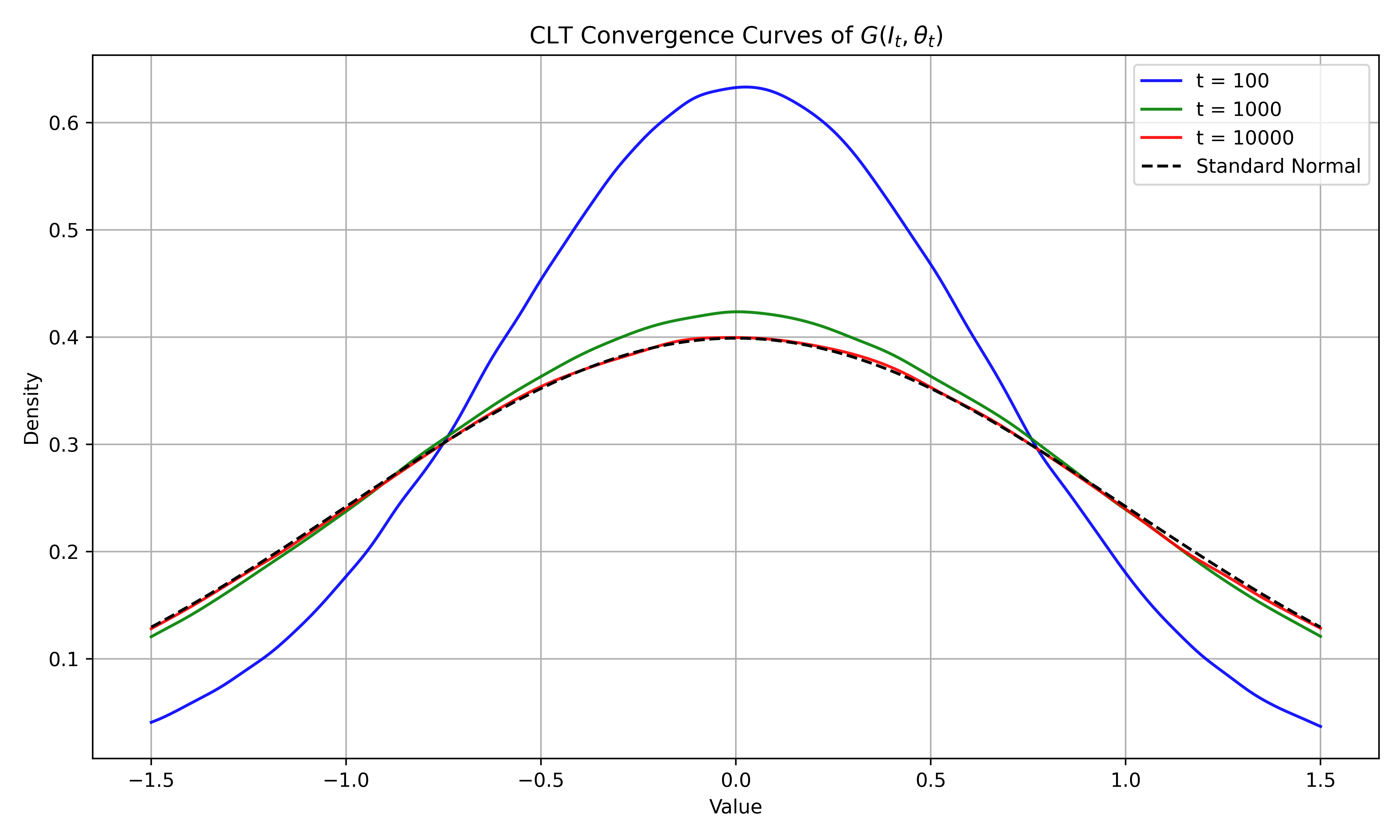}
    \subcaption{CLT convergence curves of $G(I_t,\theta_t)$}
    \label{fig:clt-convergence}
  \end{minipage}
  \caption{Two key statistical behaviors: (a) the decay of centroid norm under various stochastic intensities, and (b) the convergence of empirical distribution towards the standard normal as predicted by the Central Limit Theorem.}
  \label{fig:statistical-summary}
\end{figure}

\FloatBarrier


As observed in \autoref{fig:hamilton-phase-evolution} and \autoref{fig:hamilton-noise-comparison}, both the deterministic and stochastic systems exhibit convergence toward equilibrium over time. However, the perturbed system initially evolves at a similar or slightly slower rate compared to the deterministic case, but eventually surpasses it as time increases. Moreover, the stronger the Brownian intensity $c$, the faster the evolution.
Specifically, when $t = 1000$, the ensemble ripple remains visible for $c = 0.05$, but disappears for $c = 0.1$. For $c = 0.2$, the ripple vanishes even earlier, disappearing by $t = 100$. These observations provide strong support for the exponential convergence property described in Equation~\eqref{zhishujishoulian}. These all validate Theorem \ref{theorem3.1} and Remark \ref{remark3.2}.

To more clearly observe the above characteristics, we present the corresponding envelope curves below. However, due to changes in the Hamiltonian, it is no longer possible to derive an explicit expression for the envelope as was done in \cite{444555}. Instead, we employ numerical simulations to obtain the envelope using computational methods. Consider the envelope line corresponding to $|\langle G \rangle_t / q_0|$ as shown in \autoref{fig:centroid-norm}.

It is evident from \autoref{fig:centroid-norm} that the convergence accelerates as the noise intensity $c$ increases, particularly when $c = 0.2$.
However, we cannot directly present the Central Limit Theorem corresponding to the observable function $G$ in the form of expression $G(I,\theta)=\sqrt{2I}e^{-\mathrm{i}\theta}$. Instead, we focus solely on the coordinate $q$, that is, we redefine the observable function as $G = I \cos \theta$. Furthermore, to obtain a smoother convergence curve, we increase the number of sample points to $100,0000$. The resulting \autoref{fig:clt-convergence} exhibits convergence toward a normal distribution over time, consistent with the conclusion of the Theorem \ref{thm:CLT-IL}.

\subsection{Numerical simulation of another form of stochastic perturbations}
In this section, we briefly present the results corresponding to the random perturbation $X_t$ (see Equation \eqref{mapwithX} and Theorem \ref{theorem3.2}). 

From the following figures and \autoref{fig:hamilton-noise-comparison}, it can be observed that for stochastic perturbations with the same intensity $c$ and iteration count $t$, but different forms, the rate of evolution depends on the nature of the perturbation itself. For instance, the evolution in Figure 5.3 (a) is slower than that in Figure 5.2 (b). Moreover, for perturbations of the same form, the convergence speed is determined by the intensity of the disturbance, as illustrated in \autoref{fig:hamilton-X-comparison}.
   \begin{figure}[htbp]
  \centering
  \begin{minipage}[t]{0.48\textwidth}
    \centering
    \includegraphics[width=\textwidth]{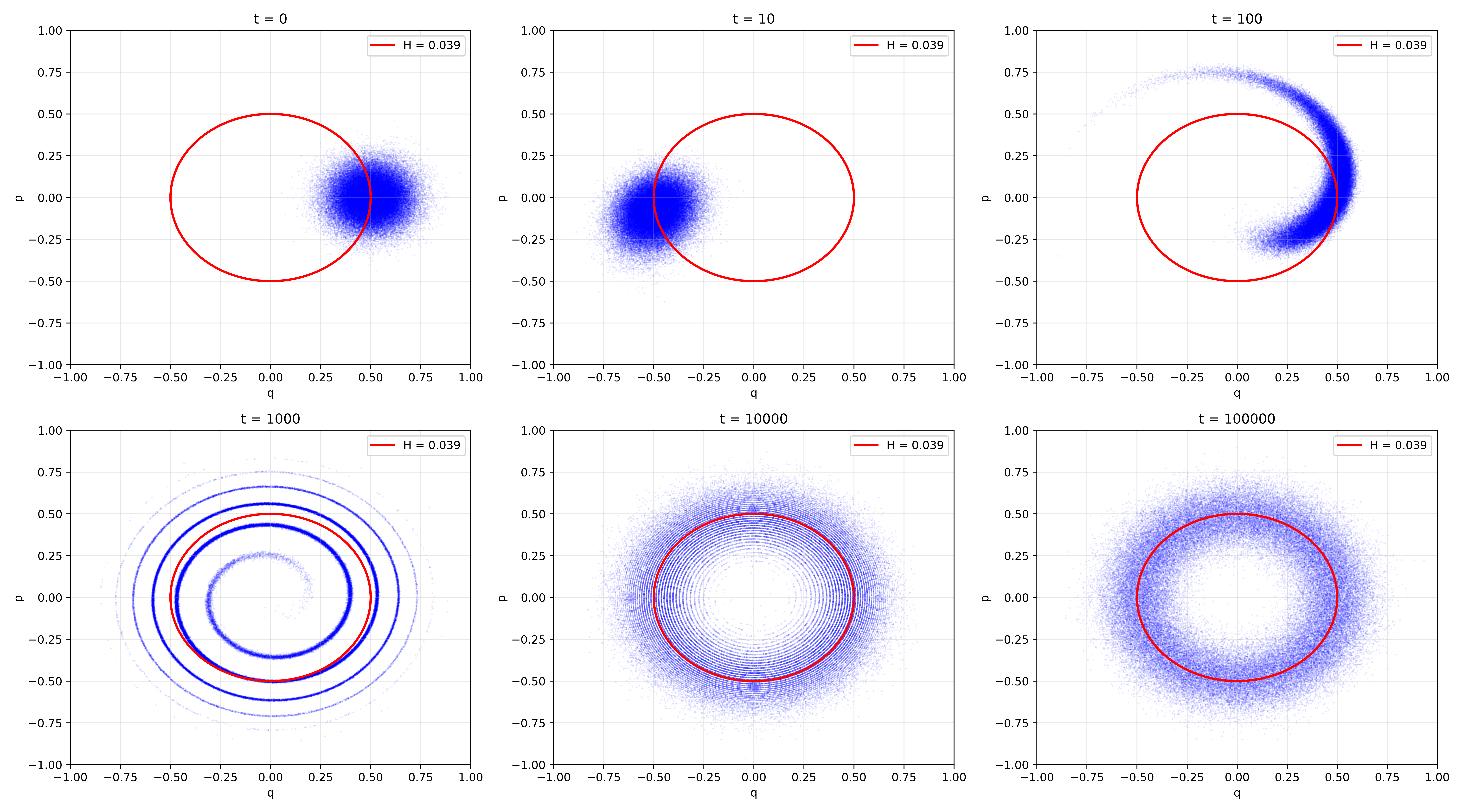}
    \subcaption{Phase space evolution under $0.1X(t)$ perturbation}
    \label{fig:hamilton-X}
  \end{minipage}
  \hfill
  \begin{minipage}[t]{0.48\textwidth}
    \centering
    \includegraphics[width=\textwidth]{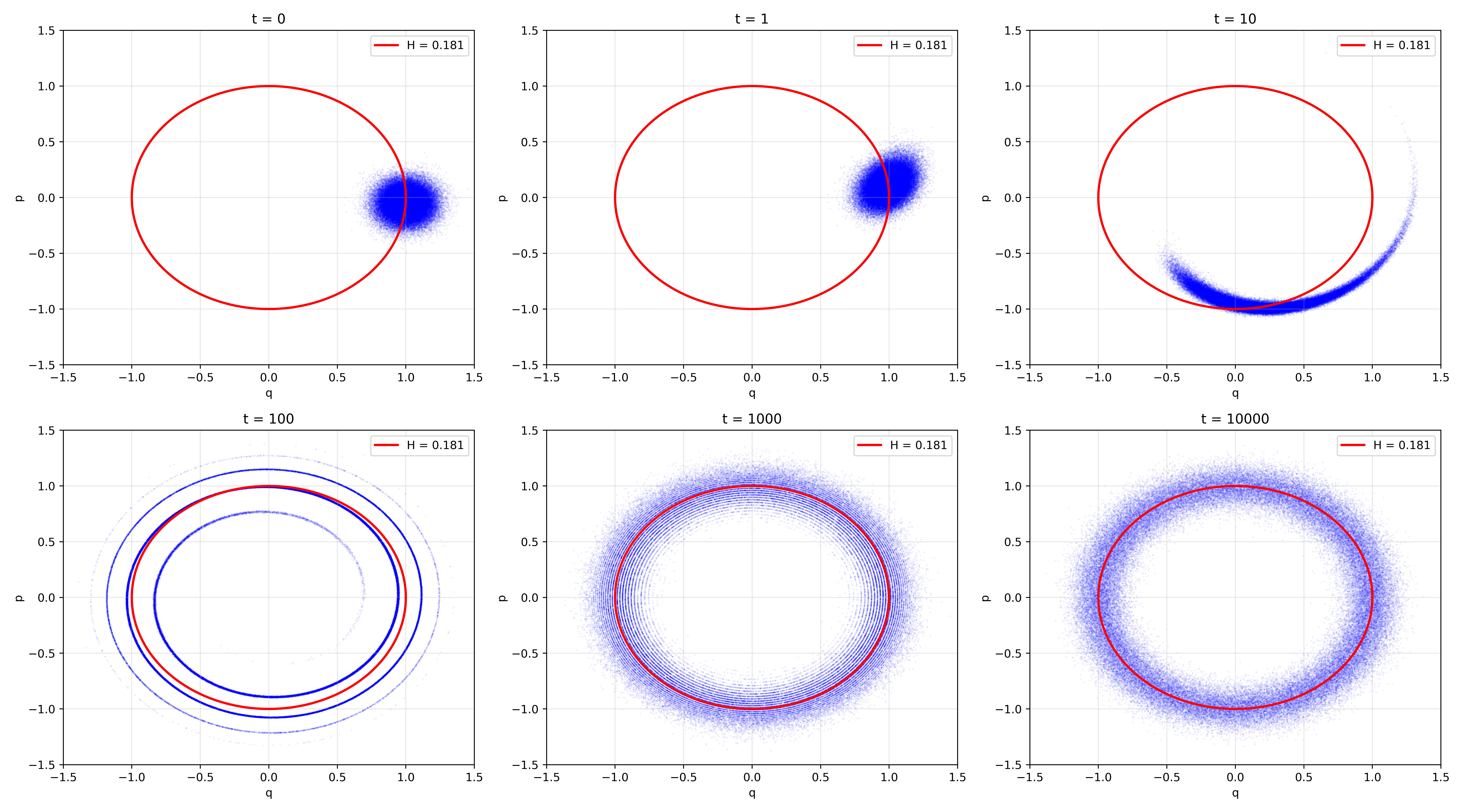}
    \subcaption{Phase space evolution under $X(t)$ perturbation}
    \label{fig:hamilton-0.1X}
  \end{minipage}
  \caption{Comparison of phase space evolution under different quasiperiodic perturbations.}
  \label{fig:hamilton-X-comparison}
\end{figure}

\endgroup

\section{Conclusion}\label{sec:conclusion}

We investigated statistical ensembles generated by iteration maps of discrete integrable Hamiltonian systems in both deterministic and stochastic settings. 
For the deterministic case, we proved a law of large numbers showing that the time-averaged ensemble converges to the initial $\theta$-average of the observable; this requires a nonresonance condition on the frequency but does not rely on the Riemann-Lebesgue lemma. 
When a Brownian perturbation is added to the angle, the same limit holds with exponentially fast decorrelation at the level of Fourier modes, and the conclusion extends to more general stochastic processes under mild Cesàro-   type decay or mixing assumptions.

For the central limit theorem, we considered
\[
X_j=G(\mathcal F^{j}(I,\theta))-\langle \bar G(I)\rangle_0,
\qquad 
\mathcal X_N=\frac{1}{\sqrt N}\sum_{j=1}^N X_j,
\]
and established a Gaussian limit under analytic-strip hypotheses for the frequency and observables, together with uniform third-moment bounds. 
The key ingredients are: exponential decay of lag-$k$ covariances, a uniform Lindeberg condition, and a controlled logarithmic expansion of the characteristic function. 
We also identified the limiting variance in the form
\[
\sigma_*^{2}=\sigma^{2}+2\sum_{k\ge1} c_k,
\]
where $\sigma^{2}$ is the variance of a single observation in the limit and $c_k$ are the Cesàro limits of the lag-$k$ covariances; in the Brownian case, these covariances decay exponentially in $k$. 
A numerical experiment corroborates the predicted convergence and limiting distribution.

\section*{Acknowledgment}
    We are indebted to the handling editor and the anonymous reviewers for their careful reading and many insightful suggestions, which have substantially improved the clarity, rigor, and scope of this work. We have learned a great deal from their feedback and look forward to studying their published contributions in greater depth.

    The Author Yong Li was supported by National Natural Science Foundation of China (12071175 , 12471183 and 12531009).

\section*{Declarations}
    {\bf Conflict of interest} The authors declare that they have no conflict of interest.






\begin{thebibliography}{10}

     \bibitem{shirikyan2006law}
Shirikyan Armen.
\newblock Law of large numbers and central limit theorem for randomly forced
  pde's.
\newblock {\em Probability theory and related fields}, 134:215--247, 2006.

\bibitem{CarmonaDelarue2018}
Ren{\'e} Carmona and Fran{\c{c}}ois Delarue.
\newblock {\em Probabilistic Theory of Mean Field Games with Applications},
  volume I--II.
\newblock Springer, Cham, 2018.

\bibitem{d2016quantum}
Luca D'Alessio, Yariv Kafri, Anatoli Polkovnikov, and Marcos Rigol.
\newblock From quantum chaos and eigenstate thermalization to statistical
  mechanics and thermodynamics.
\newblock {\em Advances in Physics}, 65(3):239--362, 2016.

\bibitem{deleporte2025central}
Alix Deleporte and Gaultier Lambert.
\newblock Central limit theorem for smooth statistics of one-dimensional free
  fermions.
\newblock {\em Journal of the London Mathematical Society}, 111(1):e70045,
  2025.

\bibitem{gillespie2013computing}
Gillespie Dirk.
\newblock Computing the partition function, ensemble averages, and density of
  states for lattice spin systems by sampling the mean.
\newblock {\em Journal of computational physics}, 250:1--12, 2013.

\bibitem{feller2021introduction}
William Feller et~al.
\newblock {\em An introduction to probability theory and its applications},
  volume 963.
\newblock Wiley New York, 1971.

\bibitem{gavalakis2024entropy}
Lampros Gavalakis and Ioannis Kontoyiannis.
\newblock Entropy and the discrete central limit theorem.
\newblock {\em Stochastic Processes and their Applications}, 170:104294, 2024.

\bibitem{horbacz2016central}
Katarzyna Horbacz.
\newblock The central limit theorem for random dynamical systems.
\newblock {\em Journal of Statistical Physics}, 164(6):1261--1291, 2016.

\bibitem{2021Central}
Yiting Li, Kevin Schnelli, and Yuanyuan Xu.
\newblock Central limit theorem for mesoscopic eigenvalue statistics of
  deformed wigner matrices and sample covariance matrices.
\newblock {\em Institute of Mathematical Statistics}, (1), 2021.

\bibitem{liu2023statistical}
Xinyu Liu and Yong Li.
\newblock Statistical ensembles in integrable hamiltonian systems with almost
  periodic transitions.
\newblock {\em arXiv preprint arXiv:2311.14248}, 2023.


\bibitem{McLeish1974}
Donald~L. McLeish.
\newblock Dependent central limit theorems and invariance principles.
\newblock {\em The Annals of Probability}, 2(4):620--628, 1974.

\bibitem{444555}
Chad Mitchell.
\newblock Weak convergence to equilibrium of statistical ensembles in
  integrable hamiltonian systems.
\newblock {\em Journal of Mathematical Physics}, 60(5), 2019.

\bibitem{DiPersioMastrogiacomo2011}
Luca~Di Persio and Enrico Mastrogiacomo.
\newblock Small noise asymptotic expansions for stochastic PDEs, I. the case of
  a dissipative polynomially bounded nonlinearity.
\newblock {\em Tohoku Mathematical Journal (2)}, 63(4):877--898, 2011.

\bibitem{qu2022law}
Ziyi Qu, Zhaojun Zong, and Feng Hu.
\newblock Law of large numbers, central limit theorem, and law of the iterated
  logarithm for bernoulli uncertain sequence.
\newblock {\em Symmetry}, 14(8):1642, 2022.

\bibitem{kuksinshirikyan2006}
Armen Shirikyan.
\newblock Law of large numbers and central limit theorem for randomly forced
  PDEs.
\newblock {\em Probability Theory and Related Fields}, 134:215--247, 2006.

\bibitem{tirnakli2007central}
Ugur Tirnakli, Christian Beck, and Constantino Tsallis.
\newblock Central limit behavior of deterministic dynamical systems.
\newblock {\em Physical Review E—Statistical, Nonlinear, and Soft Matter
  Physics}, 75(4):040106, 2007.

\bibitem{wang2025central}
Chen Wang and Yong Li.
\newblock The central limit theorems for integrable hamiltonian systems
  perturbed by white noise.
\newblock {\em Journal of Differential Equations}, 416:28--51, 2025.

    \end{thebibliography}
\end{document}